\documentclass[12pt]{article}

\usepackage{amssymb, amsmath, amsthm, amscd}

\usepackage[utf8]{inputenc}

\usepackage[all,cmtip]{xy}

\usepackage{enumitem}

\usepackage{setspace}

\usepackage{mathdots}

\usepackage[mathscr]{euscript}

\usepackage[colorlinks=true]{hyperref}

\hypersetup{colorlinks   = true}

\hypersetup{linkcolor=blue}

\usepackage{tikz}

\begin{document}

\mathchardef\mhyphen="2D

\newtheorem{The}{Theorem}[section]

\newtheorem{Lem}[The]{Lemma}

\newtheorem{Prop}[The]{Proposition}

\newtheorem{Cor}[The]{Corollary}

\newtheorem{Rem}[The]{Remark}

\newtheorem{Obs}[The]{Observation}

\newtheorem{SConj}[The]{Standard Conjecture}

\newtheorem{Titre}[The]{\!\!\!\! }

\newtheorem{Conj}[The]{Conjecture}

\newtheorem{Question}[The]{Question}

\newtheorem{Prob}[The]{Problem}

\newtheorem{Def}[The]{Definition}

\newtheorem{Not}[The]{Notation}

\newtheorem{Claim}[The]{Claim}

\newtheorem{Conc}[The]{Conclusion}

\newtheorem{Ex}[The]{Example}

\newtheorem{Fact}[The]{Fact}

\newtheorem{Formula}[The]{Formula}

\newtheorem{Formulae}[The]{Formulae}

\newtheorem{The-Def}[The]{Theorem and Definition}

\newtheorem{Prop-Def}[The]{Proposition and Definition}

\newtheorem{Lem-Def}[The]{Lemma and Definition}

\newtheorem{Cor-Def}[The]{Corollary and Definition}

\newtheorem{Conc-Def}[The]{Conclusion and Definition}

\newtheorem{Terminology}[The]{Note on terminology}

\newcommand{\C}{\mathbb{C}}

\newcommand{\R}{\mathbb{R}}

\newcommand{\N}{\mathbb{N}}

\newcommand{\Z}{\mathbb{Z}}

\newcommand{\Q}{\mathbb{Q}}

\newcommand{\Proj}{\mathbb{P}}

\newcommand{\Rc}{\mathcal{R}}

\newcommand{\Oc}{\mathcal{O}}

\newcommand{\Vc}{\mathcal{V}}

\newcommand{\Id}{\operatorname{Id}}

\newcommand{\pr}{\operatorname{pr}}

\newcommand{\rk}{\operatorname{rk}}

\newcommand{\del}{\partial}

\newcommand{\delbar}{\bar{\partial}}

\newcommand{\Cdot}{{\raisebox{-0.7ex}[0pt][0pt]{\scalebox{2.0}{$\cdot$}}}}

\newcommand\nilm{\Gamma\backslash G}

\newcommand\frg{{\mathfrak g}}

\newcommand{\fg}{\mathfrak g}

\newcommand{\Oh}{\mathcal{O}}

\newcommand{\Kur}{\operatorname{Kur}}

\newcommand\gc{\frg_\mathbb{C}}

\newcommand\slawek[1]{{\textcolor{red}{#1}}}

\newcommand\dan[1]{{\textcolor{blue}{#1}}}

\newcommand\question[1]{{\textcolor{green}{#1}}}

\def\pm{P_m(X,\omega,\chi)}

\def\fmv{F_m[\chi+i\partial\bar{\partial}v]}

\def\fmp{F_m[\chi+i\partial\bar{\partial}\varphi]}
\def\pa{\partial}

\begin{center}

  {\Large\bf $m$-Pseudo-effectivity and a Monge-Amp\`ere-Type Equation for Forms of Positive Degree}

\end{center}

\begin{center}

{\large S\l{}awomir Dinew and Dan Popovici}

\end{center}

\vspace{1ex}

\noindent{\small{\bf Abstract.} Given an $n$-dimensional compact K\"ahler manifold, we continue our study of $m$-positivity in two ways. We first propose generalisations of the notions of pseudo-effective and big Bott-Chern cohomology classes of bidegree $(1,\,1)$ by relaxing the standard positivity hypotheses to their $m$-counterparts after we have proved a Lamari-type duality lemma in bidegree $(m,\,m)$. Independently, we propose a Monge-Amp\`ere-type non-linear pde whose distinctive feature is that its solutions, if any, are forms of positive degree rather than functions. We prove a form of uniqueness for the solutions and, under the assumption that a solution exists, we give a geometric application involving the $m$-bigness notion introduced in the first part.}

\vspace{1ex}

\section{Introduction}\label{section:introd} Let $(X,\,\omega)$ be a complex Hermitian manifold with $\mbox{dim}_\C X = n$ and let $m\in\{1,\dots , n\}$. In our previous work [DP25], we studied the following variant of the notion of $m$-psh function that had been introduced by Dieu in [Die06] and subsequently studied by several authors, including Harvey and Lawson in [HL13], Verbitsky in [Ver10], Dinew in [Din22] (see references therein):

\begin{Def}\label{Def:m-semi-pos} $(1)$\, A real current $T$ of bidegree $(1,\,1)$ on $X$ is said to be {\bf $m$-semi-positive} (resp. {\bf $m$-positive}) {\bf with respect to $\omega$} if the bidegree-$(m,\,m)$-current $T\wedge\omega^{m-1}$ is {\bf strongly semi-positive (resp. strongly positive)} on $X$. In this case, we write $T\geq_{m,\,\omega}0$ (resp. $T>_{m,\,\omega}0$).

  \vspace{1ex}

$(2)$\, A holomorphic line bundle $L\longrightarrow X$ is said to be {\bf $m$-semi-positive with a $C^\infty$} (resp. {\bf singular}) {\bf metric} if there exist a $C^\infty$ Hermitian metric $\omega$ on $X$ and a $C^\infty$ (resp. singular) Hermitian fibre metric $h$ on $L$ such that the curvature form (resp. current) of the Chern connection of $(L,\,h)$ has the property: \begin{eqnarray*} i\Theta_h(L)\geq_{m,\,\omega}0.\end{eqnarray*}

\end{Def}

In $\S$\ref{section:m-psef} of the present work, building on a generalisation of Lamari's duality lemma to the arbitrary bidegree $(m,\,m)$ and its complementary bidegree $(n-m,\,n-m)$ that we give as Lemma \ref{Lem:duality_generalised-Lamari}, we generalise the notions of {\it pseudo-effective} and {\it big} cohomology classes introduced in [Dem92, Definition 1.3] and in [BDPP13] as follows:

\begin{Def}\label{Def:m-psef_introd} Let $X$ be a compact complex manifold with $\mbox{dim}_\C X = n$ and let $m\in\{1,\dots , n\}$.

\vspace{1ex}

$(1)$\, For every real current $T$ of bidegree $(1,\,1)$ such that $dT=0$ and every K\"ahler metric (if any) $\omega$ on $X$, the {\bf $m^{th}$ $[\omega]$-twisted class} of $T$ is the Bott-Chern class: \begin{eqnarray}\label{eqn:m-th-twisted_class}[T]_{BC}\wedge[\omega_{m-1}]_{BC} = [T\wedge\omega_{m-1}]_{BC}\in H^{m,\,m}_{BC}(X,\,\R).\end{eqnarray}

The Bott-Chern class $[T]_{BC}\in H^{1,\,1}_{BC}(X,\,\R)$ is said to be:

\vspace{1ex}

$\bullet$ $[\omega]$-$m$-{\bf pseudo-effective} ($[\omega]$-$m$-{\bf psef}), a fact denoted by $[T]_{BC}\wedge[\omega_{m-1}]_{BC}\geq 0$ or by $[T\wedge\omega_{m-1}]_{BC}\geq 0$, if there exists a real current $S$ of bidegree $(m-1,\,m-1)$ on $X$ such that \begin{eqnarray}\label{eqn:T_m-psef-class}T\wedge\omega_{m-1} + i\partial\bar\partial S\geq 0 \hspace{3ex} \mbox{(strongly)}\end{eqnarray} as an $(m,\,m)$-current on $X$;

\vspace{1ex}

$\bullet$ $[\omega]$-$m$-{\bf big}, a fact denoted by $[T]_{BC}\wedge[\omega_{m-1}]_{BC}> 0$ or by $[T\wedge\omega_{m-1}]_{BC}> 0$, if there exists a real current $S$ of bidegree $(m-1,\,m-1)$ on $X$ such that \begin{eqnarray*}T\wedge\omega_{m-1} + i\partial\bar\partial S\geq \varepsilon\,\gamma^m \hspace{1ex} \mbox{(strongly)}\end{eqnarray*} as an $(m,\,m)$-current on $X$, for some constant $\varepsilon>0$ and some $(1,\,1)$-form $\gamma>0$.
\vspace{1ex}

$(2)$\, Suppose that $X$ is K\"ahler. For every K\"ahler class $[\omega]_{BC}\in H^{1,\,1}_{BC}(X,\,\R)$,

\vspace{1ex}

$\bullet$ the following subset is called the {\bf $[\omega]$-$m$-pseudo-effective} ({\bf $[\omega]$-$m$-psef}) {\bf cone} of $X$ : \begin{eqnarray}\label{eqn:omega-m_psef-cone}{\cal E}_{[\omega],\,m}(X):=\bigg\{[T]_{BC}\in H^{1,\,1}_{BC}(X,\,\R)\,\mid\,[T]_{BC}\wedge[\omega_{m-1}]_{BC}\geq 0\bigg\}\subset H^{1,\,1}_{BC}(X,\,\R);\end{eqnarray}

\vspace{1ex}

$\bullet$ the following subset is called the {\bf $[\omega]$-$m$-big cone} of $X$ : \begin{eqnarray}\label{eqn:omega-m_big-cone}{\cal B}_{[\omega],\,m}(X):=\bigg\{[T]_{BC}\in H^{1,\,1}_{BC}(X,\,\R)\,\mid\,[T]_{BC}\wedge[\omega_{m-1}]_{BC}> 0\bigg\}\subset H^{1,\,1}_{BC}(X,\,\R).\end{eqnarray}

\end{Def}

\vspace{2ex}

We stress that part $(1)$ of the above definition requires the bidegree-$(m,\,m)$ cohomology class $[T\wedge\omega_{m-1}]_{BC}\in H^{m,\,m}_{BC}(X,\,\R)$ to contain a current with a certain strong positivity property even if this condition defines a property of the bidegree-$(1,\,1)$ cohomology class $[T]_{BC}\in H^{1,\,1}_{BC}(X,\,\R)$. This condition in bidegree $(m,\,m)$ was revealed to us by the duality Lemma \ref{Lem:duality_generalised-Lamari}. On the other hand, when it exists, the form $\gamma>0$ in the definition of $[\omega]$-$m$-bigness can be chosen to equal $\omega$, by compactness of $X$, after a possible adjustment of $\varepsilon>0$. We chose to keep it arbitrary so as not to give the impression that it depends on the $[\omega]$ that is part and parcel of this definition.

After proving various properties of these objects, we get the following result characterising one of these notions.

\begin{The}\label{The:m-pos_equiv_introd} Let $(X,\omega)$ be a compact K\"ahler manifold with $\mbox{dim}_\C X = n$. Fix $m\in\{1,\dots , n\}$.

  \vspace{1ex}

Then, for every cohomology class $\mathfrak{c}\in H^{1,\,1}_{BC}(X,\,\R)$, the following two statements are equivalent:

  \vspace{1ex}

  (a) There exist a $C^\infty$ $d$-closed real $(1,\,1)$-form $\alpha\in\mathfrak{c}$ and a form $\Omega_0\in C^\infty_{n-m,\,n-m}(X,\,\R)$ with the properties: \begin{eqnarray}\label{eqn:m-pos_equiv_Omega-prop}\Omega_0>0 \hspace{2ex} (\mbox{weakly}) \hspace{3ex}\mbox{and}\hspace{3ex} \partial\bar\partial\Omega_0 = 0\end{eqnarray} such that $\alpha\wedge\omega^{m-1}\wedge\Omega_0 >0$ everywhere on $X$;

 \vspace{1ex}

 (b) The cohomology class $-\mathfrak{c}\in H^{1,\,1}_{BC}(X,\,\R)$ is {\bf not $[\omega]$-$m$-pseudo-effective}.

\end{The}

\vspace{2ex}

In $\S$\ref{section:equation}, we propose a new non-linear partial differential equation of Monge-Amp\`ere type with a peculiarity.

As far as we are aware, the solutions, if any, of all the Monge-Amp\`ere-type equations (all of which are highly non-linear PDEs) in the literature are (smooth or otherwise) {\it functions}. For geometric and analytic reasons, some of which will be spelt out further down the road, we will make a case for this type of PDEs whose solutions are differential forms of bidegree $(r,\,r)$ with $r$ allowed to be $>0$.

Specifically, in keeping with the set-up of this paper and building on the experience of [Pop15] where a Monge-Amp\`ere-type pde with function solutions was introduced through a condition on $(n-1,\,n-1)$-forms, we propose the following Monge-Amp\`ere-type equation for positive-degreed form solutions:

\begin{Def}\label{Def:equation} Let $X$ be a compact complex manifold with $\mbox{dim}_\C X = n$ and let $m\in\{1,\dots , n\}$. Suppose there exists a {\bf K\"ahler} metric $\omega$ on $X$.

  Given a volume form $dV\in C^\infty_{n,\,n}(X,\,\R)$ with $dV>0$ and a form $\alpha\in C^\infty_{m,\,m}(X,\,\R)$ such that $d\alpha = 0$ and $\alpha>0$ strongly on $X$, we consider the equation: \begin{eqnarray}\label{eqn:equation}\bigg[\star_\omega\bigg((\alpha + i\partial\bar\partial u)\wedge\omega_{n-m-1}\bigg)\bigg]^n = dV,\end{eqnarray} whose solutions $u\in C^\infty_{m-1,\,m-1}(X,\,\R)$, if any, are subject to the initial conditions: \begin{eqnarray}\label{eqn:equation_initial-conditions}\alpha + i\partial\bar\partial u > 0 \hspace{1ex}\mbox{(strongly)} \hspace{5ex}  \mbox{and} \hspace{5ex} u\in\ker\partial^\star_\omega\cap\ker\bar\partial^\star_\omega.\end{eqnarray}

\end{Def}

While we hope to be able to take up the problem of the existence of a solution to this equation in future work, our arguments in this present paper lead to a proof (cf. Theorem \ref{The:equation_uniqueness}) of the uniqueness, up to a constant multiple of $\omega^{m-1}$, of the solution of equation (\ref{eqn:equation}) whose primitive components of positive degree in the Lefschetz decomposition (\ref{eqn:Lefschetz-decomp}) have been prescribed. This prescription can be thought of as a kind of ``boundary'' conditions.

Actually, our method gives the uniqueness, up to an additive constant, of the primitive component of degree $0$ of the solution in the more general setting (cf. Theorem \ref{The:equation_uniqueness_introd} just below) where only the primitive component of bidegree $(1,\,1)$ of the solution has been prescribed.

\begin{The}\label{The:equation_uniqueness_introd} The set-up is the one described in Definition \ref{Def:equation}. If there exist two solutions $u_1,\,u_2\in C^\infty_{m-1,\,m-1}(X,\,\R)$ of equation (\ref{eqn:equation}) satisfying the initial conditions (\ref{eqn:equation_initial-conditions}) such that their $\omega$-primitive components $(u_1)_{prim}^{(m-2)}$ and $(u_2)_{prim}^{(m-2)}$ of bidegree $(1,\,1)$ in their Lefschetz decompositions (\ref{eqn:Lefschetz-decomp}) are prescribed arbitrarily by a same $(1,\,1)$-form, then there exists a constant $C\in\R$ such that \begin{eqnarray*}(u_1)_{prim}^{(m-1)} = (u_2)_{prim}^{(m-1)} + C.\end{eqnarray*}

\end{The}

The preliminaries to the proof of this result include various formulae (cf. Lemmas \ref{Lem:Laplacians_link} and \ref{Lem:star_product_zeta_omega-powers}, Corollary and Definition \ref{Cor-Def:P_omega_def-formula}) involving Laplacians and the Hodge star operator that we prove using (K\"ahler and Hermitian) commutation relations, primitive forms and the Lefschetz decomposition.

\vspace{2ex}

In $\S$\ref{section:application_PDE}, we give an application to the notions introduced in $\S$\ref{section:m-psef} and to the equation introduced in $\S$\ref{section:equation}. Specifically, supposing that equation (\ref{eqn:equation}) admits a solution satisfying certain initial conditions (see Definition \ref{Def:existence_equation_initial-cond}), we give a higher-degree analogue of the main result of [Pop16] that solved at the time on compact K\"ahler (and slightly more general) manifolds the qualitative part of the version for differences of two nef classes of Demailly's {\it Transcendental Morse Inequalities Conjecture} as formulated, for example, in [BDPP13, (ii) of Conjecture 10.1].

To state this result, we consider the following pointwise linear map induced by an arbitrary Hermitian metric $\omega$ on an $n$-dimensional complex manifold $X$ and by a given $m\in\{1,\dots , n-1\}$: \begin{eqnarray*}\Lambda^{m,\,m}T^\star X\ni\alpha\longmapsto \alpha_\omega:=\star_\omega\bigg(\alpha\wedge\omega_{n-m-1}\bigg)\in \Lambda^{1,\,1}T^\star X,\end{eqnarray*} where $\star_\omega$ is the Hodge star operator associated with $\omega$.

Our set-up is the following:

\vspace{1ex}

$\bullet$ $(X,\,\omega)$ is a compact {\bf K\"ahler} manifold with $\mbox{dim}_\C X = n$;

$\bullet$ $m$ is a given integer lying in $\{1,\dots , n\}$;

$\bullet$ $\alpha, \beta\in C^\infty_{m,\,m}(X,\,\R)$ are differential forms supposed to exist such that:

\vspace{1ex}

(i)\, $d\alpha = d\beta = 0$; (ii)\, $\alpha>0$ (strongly) and $\beta\geq C\,\omega_m$ (strongly) for some constant $C>0$;

(iii)\, $\Delta''_\omega\alpha_\omega =0$, where $\Delta''_\omega:=\bar\partial\bar\partial^\star + \bar\partial^\star\bar\partial$ is the $\bar\partial$-Laplacian induced by $\omega$.

\vspace{1ex}

We obtain the following result (see Theorem \ref{The:current_existence} for a slightly more general statement):

\begin{The}\label{The:current_existence_introd} Suppose that equation (\ref{eqn:equation}) is {\bf solvable} in the sense of Definition \ref{Def:existence_equation_initial-cond}. Then, in the set-up described above, if the forms $\alpha$ and $\beta$ satisfy the condition: \begin{eqnarray*}\displaystyle \frac{1}{(n-m)!}\,\int\limits_X(\alpha_\omega)^{n-m}\wedge\beta < \frac{1}{n!}\,\int\limits_X(\alpha_\omega)^n,\end{eqnarray*} there exists a $d$-closed real current $T$ of bidegree $(m,\,m)$ on $X$ such that: \begin{eqnarray*}(i)\, T\geq \delta\,(\alpha_\omega)_m \hspace{3ex}\mbox{for some constant}\hspace{1ex} \delta>0; \hspace{2ex}   (ii)\, T\in\bigg[(\alpha_\omega)_m - \beta\bigg]_{BC},\end{eqnarray*} where $[\,\cdot\,]_{BC}$ stands for the Bott-Chern cohomology class of the specified form.

\end{The}

\vspace{2ex}

This result connects with the notions introduced in $\S$\ref{section:m-psef} through the following

\begin{Cor}\label{Cor:n-1_current_existence_big_introd} When $m=n-1$ in the setting of Theorem \ref{The:current_existence_introd}, the inverse image of the class $[(\star_\omega\alpha)_{n-1} - \beta]_{BC}\in H^{n-1,\,n-1}_{BC}(X,\,\R)$ under the Hard Lefschetz isomorphism \begin{eqnarray*}H^{1,\,1}_{BC}(X,\,\R)\ni[a]_{BC}\longmapsto[\omega_{n-2}\wedge a]_{BC}\in H^{n-1,\,n-1}_{BC}(X,\,\R)\end{eqnarray*} is an $[\omega]$-$(n-1)$-{\bf big} class.

\end{Cor}

\vspace{2ex}

This corollary is the direct counterpart in the case $m=n-1$ of the main result obtained in [Pop16] for the classical case where $m=1$.

\vspace{2ex}

\noindent {\bf Acknowledgments}. The first-named author was partially supported by grant no. 2021/41/B/ST1/01632 from the National Science Center, Poland.

\vspace{2ex}

\section{$m$-pseudo-effectivity}\label{section:m-psef} In this section, we generalise the classical notions of pseudo-effective (psef) and big cohomology classes (and, more particularly, holomorphic line bundles) by relaxing standard positivity notions for forms and currents to their $m$-(semi-)positive counterparts.

\subsection{A duality phenomenon}\label{subsection:duality} The starting point is the following generalisation to arbitrary bidegrees $(m,\,m)$ of Lamari's duality lemma [Lam99, Lemme 3.3] for bidegree $(1,\,1)$.

\begin{Lem}\label{Lem:duality_generalised-Lamari} Let $X$ be a compact complex manifold with $\mbox{dim}_\C X = n$, let $m\in\{1,\dots , n\}$ and let $\theta\in C_{m,\,m}^\infty(X,\,\R)$. The following statements are equivalent.

\vspace{1ex}

(i)\, There exists a real current $S$ of bidegree $(m-1,\,m-1)$ on $X$ such that the $(m,\,m)$-current $\theta + i\partial\bar\partial S$ is strongly semi-positive (a fact denoted by $\theta + i\partial\bar\partial S\geq 0$ (strongly)) on $X$.

\vspace{1ex}

(ii)\, $\int\limits_X\theta\wedge\Omega \geq 0$ for every $\Omega\in C_{n-m,\,n-m}^\infty(X,\,\R)$ such that $\partial\bar\partial\Omega = 0$ and $\Omega>0$ (weakly) on $X$.

\end{Lem}

\noindent {\it Proof.} $(i)\implies(ii)$ This implication is immediate since $(\theta + i\partial\bar\partial S)\wedge\Omega$ is a semi-positive $(n,\,n)$-current on $X$ (as the product of a strongly semi-positive current by a weakly strictly positive $C^\infty$ form), hence we have the inequality below:  \begin{eqnarray*}0\leq\int\limits_X(\theta + i\partial\bar\partial S)\wedge\Omega = \int\limits_X\theta\wedge\Omega,\end{eqnarray*} where the equality follows from the Stokes theorem and the hypothesis $\partial\bar\partial\Omega = 0$.

\vspace{1ex}

$(ii)\implies(i)$ We consider the open convex subset: \begin{eqnarray*}U:=\bigg\{\Omega\in C_{n-m,\,n-m}^\infty(X,\,\R) \,\mid\, \Omega>0 \hspace{1.5ex} \mbox{(weakly)}\bigg\}\subset C_{n-m,\,n-m}^\infty(X,\,\R)\end{eqnarray*} and the $\R$-linear map: \begin{eqnarray}\label{eqn:linear-map}\int\limits_X\theta\wedge\cdot: C_{n-m,\,n-m}^\infty(X,\,\R)\cap\ker(\partial\bar\partial)\longrightarrow\R,   \hspace{5ex} \Omega\longmapsto\int\limits_X\theta\wedge\Omega.\end{eqnarray} Thanks to the hypothesis (ii), the restriction of this map to $U\cap\ker(\partial\bar\partial)$ is non-negative: \begin{eqnarray*}\bigg(\int\limits_X\theta\wedge\cdot\,\bigg)_{|U\cap\ker(\partial\bar\partial)}\geq 0.\end{eqnarray*}

Therefore, there are two possible cases.

\vspace{1ex}

$\bullet$ {\it Case $1$.} This is the case where there exists $\Omega_0\in U\cap\ker(\partial\bar\partial)$ such that $\int\limits_X\theta\wedge\Omega_0 = 0$.

\vspace{1ex}

In this case, we get $\int\limits_X\theta\wedge\Omega = 0$ for all $\Omega\in C_{n-m,\,n-m}^\infty(X,\,\R)\cap\ker(\partial\bar\partial)$ (i.e. the map (\ref{eqn:linear-map}) vanishes identically). This is equivalent to having $\theta\in\mbox{Im}\,(\partial\bar\partial)$, hence to the existence of a real $(m-1,\,m-1)$-form $S$ on $X$ such that $\theta + i\partial\bar\partial S = 0$. Thus, (i) holds in this case.

To prove the claim that the map (\ref{eqn:linear-map}) vanishes identically in this case, we fix $\Omega\in C_{n-m,\,n-m}^\infty(X,\,\R)\cap\ker(\partial\bar\partial)$ with $\Omega\neq\Omega_0$ and we consider the $\R$-line \begin{eqnarray*}\Omega_t:=(1-t)\,\Omega_0 + t\,\Omega\in C_{n-m,\,n-m}^\infty(X,\,\R)\cap\ker(\partial\bar\partial),   \hspace{5ex} t\in\R,\end{eqnarray*} and the affine function \begin{eqnarray*}f:\R\longrightarrow\R, \hspace{5ex} f(t):=\int\limits_X\theta\wedge\Omega_t.\end{eqnarray*} Thus, there exist constants $a,b\in\R$ such that $f(t) = at+b$ for all $t\in\R$. Our hypothesis of {\it Case $1$} means that $f(0)=0$, hence $b=0$ and $f(t) = at$ for all $t\in\R$.

Now, $U\cap\ker(\partial\bar\partial)$ is open in $C_{n-m,\,n-m}^\infty(X,\,\R)\cap\ker(\partial\bar\partial)$ and $\Omega_0\in U\cap\ker(\partial\bar\partial)$. Therefore, $\Omega_t\in U\cap\ker(\partial\bar\partial)$ for all $t\in\R$ sufficiently close to $0$. This means that $\Omega_t > 0$ (weakly) and $\partial\bar\partial\Omega_t = 0$ for all $t\in\R$ sufficiently close to $0$. Hypothesis (ii) implies that \begin{eqnarray*}at = f(t)= \int\limits_X\theta\wedge\Omega_t\geq 0   \hspace{5ex} \mbox{for all}\hspace{1ex} t\in\R \hspace{1ex}\mbox{sufficiently close to}\hspace{1ex} 0.\end{eqnarray*} This implies that $a=0$ (which amounts to $f\equiv 0$, which implies that $f(1)=0$, as claimed) since otherwise $f(t)$ would change signs when $t$ switches from being negative to being positive.

\vspace{1ex}

$\bullet$ {\it Case $2$.} This is the case where $(\int\limits_X\theta\wedge\cdot\,)_{|U\cap\ker(\partial\bar\partial)}> 0$, namely the case where $\int\limits_X\theta\wedge\Omega > 0$ for all $\Omega\in C_{n-m,\,n-m}^\infty(X,\,\R)$ such that $\Omega>0$ (weakly) and $\partial\bar\partial\Omega = 0$.

\vspace{1ex}

We consider the vector subspace of codimension $1$ ($=$ the kernel of the linear map (\ref{eqn:linear-map})): \begin{eqnarray*}F:=\bigg\{\Omega\in C_{n-m,\,n-m}^\infty(X,\,\R)\,\mid\, \partial\bar\partial\Omega = 0 \hspace{1ex}\mbox{and}\hspace{1ex} \int\limits_X\theta\wedge\Omega = 0\bigg\}\subset C_{n-m,\,n-m}^\infty(X,\,\R)\cap\ker(\partial\bar\partial).\end{eqnarray*}

Our hypothesis of {\it Case $2$} ensures that $U\cap F = \emptyset$. This property, together with $U$ being a non-empty open convex subset and $F$ being a vector subspace of the locally convex space $C_{n-m,\,n-m}^\infty(X,\,\R)$, ensures, thanks to the Hahn-Banach separation theorem, the existence of continuous linear functional
\begin{eqnarray*}\widetilde{T}: C_{n-m,\,n-m}^\infty(X,\,\R)\longrightarrow\R
\end{eqnarray*} (i.e. of a real current $\widetilde{T}$ of bidegree $(m,\,m)$ on $X$) that separates $U$ and $F$ in the sense that:
\begin{eqnarray*}\widetilde{T}_{|U} > 0  \hspace{5ex}\mbox{and}\hspace{5ex} \widetilde{T}_{|F}\equiv 0.\end{eqnarray*}
Note that the property $\widetilde{T}_{|U} > 0$ means that the current $\widetilde{T}$ is strongly semi-positive (i.e. $\widetilde{T}\geq 0$ strongly) on $X$ thanks to the duality between weakly strictly positive $C^\infty$ forms and strongly semi-positive currents of the complementary bidegree. In particular, $\widetilde{T}$ is an order-zero current and hence, for any form $\Omega\in C_{n-m,\,n-m}^\infty(X,\,\R)$, we can write $T(\Omega)=\int_X\widetilde{T}\wedge\Omega $.

  Let $\Omega_1\in U\cap\ker(\partial\bar\partial)$. Then, $\Omega_1\notin F$ and $\int_X\theta\wedge\Omega_1 >0$ and $\int_X\widetilde{T}\wedge\Omega_1 >0$. Thus, there exists a constant $\lambda\in(0,\,\infty)$ such that \begin{eqnarray*}\int_X\theta\wedge\Omega_1 = \lambda\,\int_X\widetilde{T}\wedge\Omega_1,\end{eqnarray*} a fact that amounts to \begin{eqnarray*}\bigg(\int\limits_X(\theta - \lambda\,\widetilde{T})\wedge\cdot\,\bigg)_{|\R\,\Omega_1}\equiv 0.\end{eqnarray*}

  On the other hand, by construction, we also have: \begin{eqnarray*}\bigg(\int\limits_X(\theta - \lambda\,\widetilde{T})\wedge\cdot\,\bigg)_{|F}\equiv 0.\end{eqnarray*}

  Since $C_{n-m,\,n-m}^\infty(X,\,\R)\cap\ker(\partial\bar\partial) = \R\,\Omega_1\oplus F$, we infer that \begin{eqnarray*}\bigg(\int\limits_X(\theta - \lambda\,\widetilde{T})\wedge\cdot\,\bigg)_{|C_{n-m,\,n-m}^\infty(X,\,\R)\cap\ker(\partial\bar\partial)}\equiv 0.\end{eqnarray*} This is equivalent to having $\theta - \lambda\,\widetilde{T}\in\mbox{Im}\,(\partial\bar\partial)$, hence to the existence of a real current $S$ of bidegree $(m-1,\,m-1)$ on $X$ such that $\Theta - \lambda\,\widetilde{T} = -i\partial\bar\partial S$.

  Summing up, we have got: \begin{eqnarray*}\Theta + i\partial\bar\partial S = \lambda\,\widetilde{T}\geq 0  \hspace{2ex} \mbox{(strongly)}.\end{eqnarray*} This proves (i).  \hfill $\Box$

\vspace{3ex}

  The duality Lemma \ref{Lem:duality_generalised-Lamari} prompts the introduction of the cones described below, the second one being a generalisation of the Gauduchon cone introduced in [Pop15].

  \begin{Prop-Def}\label{Prop-Def:F_m-G_m_cones} Let $X$ be a compact complex manifold with $\mbox{dim}_\C X = n$ and let $m\in\{1,\dots , n\}$.

\vspace{1ex}

$(1)$\, The subset: \begin{eqnarray*}{\cal F}_m(X) & := & \bigg\{[T]_{BC}\in H^{m,\,m}_{BC}(X,\,\R)\,\mid\,T\geq 0 \hspace{2ex}\mbox{strongly as an $(m,\,m)$-current on $X$} \\
  & &  \hspace{40ex}  \mbox{and}\hspace{2ex} dT=0\bigg\}\subset H^{m,\,m}_{BC}(X,\,\R) \end{eqnarray*} is a convex cone. If, moreover, there exists a Hermitian metric $\gamma$ on $X$ such that $\partial\bar\partial\gamma^{n-m} = 0$, this cone is {\bf closed}.

\vspace{1ex}

$(2)$\, The subset: \begin{eqnarray*}{\cal G}_m(X) & := & \bigg\{[\Omega]_A\in H^{n-m,\,n-m}_A(X,\,\R)\,\mid\,\Omega\in C^\infty_{n-m,\,n-m}(X,\,\R) \hspace{2ex}\mbox{such that}\hspace{2ex}   \Omega> 0 \hspace{2ex}\mbox{(weakly)} \\
  & &  \hspace{40ex} \mbox{and}\hspace{2ex} \partial\bar\partial\Omega=0\bigg\}\subset H^{n-m,\,n-m}_A(X,\,\R)\end{eqnarray*} is an {\bf open} convex cone.

\vspace{1ex}

$(3)$\, If there exists a Hermitian metric $\gamma$ on $X$ such that $\partial\bar\partial\gamma^{n-m} = 0$, the cone ${\cal F}_m(X)\subset H^{m,\,m}_{BC}(X,\,\R)$ and the closure $\overline{{\cal G}_m(X)}\subset H^{n-m,\,n-m}_A(X,\,\R)$ of the cone ${\cal G}_m(X)$ are {\bf dual} to each other under the duality between the Bott-Chern cohomology of bidegree $(m,\,m)$ and the Aeppli cohomology of the complementary bidegree $(n-m,\,n-m)$.

\end{Prop-Def}

  \noindent {\it Proof.} $(1)$-$(2)$ We will only prove the closedness of ${\cal F}_m(X)$ under the stated assumption, the other statements being left to the reader. Let  \begin{eqnarray*}{\cal F}_m(X)\ni[T_j]_{BC}\underset{j\to\infty}{\longrightarrow}\mathfrak{c}\in H^{m,\,m}_{BC}(X,\,\R)\simeq\bigg(H^{n-m,\,n-m}_A(X,\,\R)\bigg)^\star
  \end{eqnarray*} be a convergent sequence of cohomology classes represented by $(m,\,m)$-currents $T_j$ such that \begin{eqnarray*}T _j\geq 0 \hspace{2ex}\mbox{strongly on}\hspace{1ex} X \hspace{5ex}\mbox{and}\hspace{5ex} dT_j=0\end{eqnarray*} for all $j$. We have to show that $\mathfrak{c}\in{\cal F}_m(X)$.

      From the hypotheses, we get: \begin{eqnarray*}0\leq\int\limits_X T_j\wedge\gamma^{n-m} = [T_j]_{BC}.[\gamma^{n-m}]_A\underset{j\to\infty}{\longrightarrow}\mathfrak{c}.[\gamma^{n-m}]_A.\end{eqnarray*} The above integral constitutes the $\gamma^{n-m}$-mass of the strongly semi-positive $(m,\,m)$-current $T_j$, so the convergence implies that the sequence $(T_j)_j$ of strongly semi-positive $(m,\,m)$-currents is uniformly bounded in mass. Therefore, there exists a weakly convergent subsequence $T_{j_\nu}\underset{\nu\to\infty}{\longrightarrow}T$ whose limiting $(m,\,m)$-current is necessarily strongly semi-positive and satisfies $dT=0$.

      We deduce the convergence $[T_{j_\nu}]_{BC}\underset{\nu\to\infty}{\longrightarrow}[T]_{BC}$, hence $\mathfrak{c} = [T]_{BC}\in{\cal F}_m(X)$.

\vspace{1ex}

$(3)$\, This is a rewording of the duality Lemma \ref{Lem:duality_generalised-Lamari} under the specified assumptions when the form $\theta$ therein is supposed $d$-closed.      \hfill $\Box$

\subsection{$m$-positivity notions for cohomology classes}\label{subsection:m-positivity-notions}
  
 The duality Lemma \ref{Lem:duality_generalised-Lamari} and Definition \ref{Def:m-semi-pos} of $m$-semi-positivity prompt between them the introduction of the notions described in $(1)$ of Definition \ref{Def:m-psef_introd}. Their integral-class versions are described in the following

\begin{Def}\label{Def:m-psef} Let $L\longrightarrow X$ be a holomorphic line bundle over a compact complex manifold with $\mbox{dim}_\C X = n$. Let $m\in\{1,\dots , n\}$.

    \vspace{1ex}

    $(1)$\, We say that $L$ is {\bf $m$-pseudo-effective} ({\bf $m$-psef} for short), respectively {\bf $m$-big}, if there exist:

    \vspace{1ex}

    $\bullet$ a (possibly singular) Hermitian fibre metric $h$ on $L$;

    $\bullet$ a Hermitian metric $\omega$ on $X$;

    $\bullet$ a real current $S$ of bidegree $(m-1,\,m-1)$ on $X$;

    \vspace{1ex}

    \noindent such that \begin{eqnarray*}\frac{i}{2\pi}\Theta_h(L)\wedge\omega_{m-1} + i\partial\bar\partial S\geq 0 \hspace{3ex} \mbox{(strongly)},\end{eqnarray*} as an $(m,\,m)$-current on $X$, respectively \begin{eqnarray*}\frac{i}{2\pi}\Theta_h(L)\wedge\omega_{m-1} + i\partial\bar\partial S\geq \varepsilon\,\gamma^m \hspace{1ex} \mbox{(strongly)}\end{eqnarray*} as an $(m,\,m)$-current on $X$, for some constant $\varepsilon>0$ and some $(1,\,1)$-form $\gamma>0$.

\vspace{1ex}

If a K\"ahler class $[\omega]_{BC}\in H^{1,\,1}_{BC}(X,\,\R)$ has been fixed, $L$ is said to be {\bf $[\omega]$-$m$-pseudo-effective} ({\bf $[\omega]$-$m$-psef}), respectively {\bf $[\omega]$-$m$-big}, if there exist $h$ and $S$ as above satisfying the first, respectively second, strong semi-positivity property for some (hence any) choice of K\"ahler metric $\omega$ in the given K\"ahler class $[\omega]_{BC}$.

\vspace{1ex}

$(2)$\, For every K\"ahler metric (if any) $\omega$ on $X$, the {\bf $m^{th}$ $[\omega]$-twisted Chern class} of $L$ is the Bott-Chern class: \begin{eqnarray}\label{eqn:m-th-twisted_BC-class}c_1(L)_{BC}\wedge[\omega_{m-1}]_{BC} = \bigg[\frac{i}{2\pi}\Theta_h(L)\wedge\omega_{m-1}\bigg]_{BC}\in H^{m,\,m}_{BC}(X,\,\R),\end{eqnarray} where $c_1(L)_{BC} = \bigg[\frac{i}{2\pi}\Theta_h(L)\bigg]_{BC}\in H^{1,\,1}_{BC}(X,\,\R)$ is the usual first Chern class in the Bott-Chern cohomology of $L$.

\end{Def}

\vspace{3ex}

The first properties of these objects and their transcendental counterparts defined in the introduction are collected in the following

\begin{Lem}\label{Lem:m-psef_properties} The set-up is the one in Definitions \ref{Def:m-psef_introd} and \ref{Def:m-psef}.

\vspace{1ex}

(i)\, For every K\"ahler class $[\omega]_{BC}\in H^{1,\,1}_{BC}(X,\,\R)$, every $m\in\{1,\dots , n\}$ and every real current $T$ of bidegree $(1,\,1)$ such that $dT=0$, the $[\omega]$-$m$-semi-positivity property $[T]_{BC}\wedge[\omega_{m-1}]_{BC}\geq 0$ is independent of the choice of representative of the Bott-Chern class $[T]_{BC}\in H^{1,\,1}_{BC}(X,\,\R)$ and of the choice of K\"ahler metric in $[\omega]_{BC}$.

\vspace{1ex}

(ii)\, For every K\"ahler class $[\omega]_{BC}\in H^{1,\,1}_{BC}(X,\,\R)$ and every $m\in\{1,\dots , n\}$, a holomorphic line bundle $L\longrightarrow X$ is {\bf $[\omega]$-$m$-pseudo-effective} if and only if $c_1(L)_{BC}\in{\cal E}_{[\omega],\,m}(X)$.

\vspace{1ex}

 (iii)\, For every K\"ahler class $[\omega]_{BC}\in H^{1,\,1}_{BC}(X,\,\R)$ and every $m\in\{1,\dots , n\}$, the subset ${\cal E}_{[\omega],\,m}(X)\subset H^{1,\,1}_{BC}(X,\,\R)$ is a {\bf convex cone} in the sense that it is stable under addition and multiplication by non-negative scalars. Moreover, we have: \begin{eqnarray*}{\cal E}(X) = {\cal E}_{[\omega],\,1}(X)\subset\dots\subset{\cal E}_{[\omega],\,m}(X)\subset{\cal E}_{[\omega],\,m+1}(X)\subset\dots\subset{\cal E}_{[\omega],\,n}(X),\end{eqnarray*} where ${\cal E}(X)$ is the pseudo-effective cone of $X$ consisting of the Bott-Chern cohomology classes $[T]_{BC}$ of the $d$-closed semi-positive $(1,\,1)$-currents $T\geq 0$ on $X$.

\end{Lem}

  \noindent {\it Proof.} (i)\, If $\varphi:X\longrightarrow\R\cup\{-\infty\}$ is an $L^1_{loc}$ function and $S$ is any real current of bidegree $(m-1,\,m-1)$ on $X$, then $(T+i\partial\bar\partial\varphi)\wedge\omega_{m-1} + i\partial\bar\partial(S-\varphi\,\omega_{m-1}) = T\wedge\omega_{m-1} + i\partial\bar\partial S$ thanks to the closedness of $\omega$. Thus, the pair $(T,\,S)$ satisfies the semi-positivity property (\ref{eqn:T_m-psef-class}) if and only if the pair $(T+i\partial\bar\partial\varphi,\,S-\varphi\,\omega_{m-1})$ does.

  The independence of the choice of K\"ahler metric in a given K\"ahler class can be proved similarly.

  (ii) follows at once from the definitions.

  (iii)\, The fact that ${\cal E}_{[\omega],\,m}(X)$ is a convex cone and the equality ${\cal E}(X) = {\cal E}_{[\omega],\,1}(X)$ follow at once from the definitions. To prove the inclusion ${\cal E}_{[\omega],\,m}(X)\subset{\cal E}_{[\omega],\,m+1}(X)$, we simply notice the validity of the implication: \begin{eqnarray*}T\wedge\omega_{m-1} + i\partial\bar\partial S\geq 0 \hspace{3ex} \mbox{(strongly)} \implies \omega\wedge\bigg(T\wedge\omega_{m-1} + i\partial\bar\partial S\bigg)\geq 0 \hspace{3ex} \mbox{(strongly)}\end{eqnarray*} due to the general fact that the product of two strongly semi-positive objects (two forms or a form and a current) is again a strongly semi-positive object. Moreover, since $d\omega = 0$, we have $\omega\wedge i\partial\bar\partial S = i\partial\bar\partial(\omega\wedge S)$. \hfill $\Box$

  \vspace{2ex}

  We will now further study the convex cones ${\cal E}_{[\omega],\,m}(X)$.

  \begin{Lem}\label{Lem:limiting-S-current} Let $(X,\,\omega)$ be a compact Hermitian $n$-dimensional manifold and let $m\in\{1,\dots , n\}$. Suppose there exists a Hermitian metric $\gamma$ on $X$ such that $\partial\bar\partial\gamma^{n-m} = 0$.

    Then, for any sequences $(T_j)_{j\in\N}$ of real $(1,\,1)$-currents and $(S_j)_{j\in\N}$ of real $(m-1,\,m-1)$-currents on $X$ such that \begin{eqnarray*}dT_j = 0 \hspace{5ex}\mbox{and}\hspace{5ex} T_j\wedge\omega_{m-1} + i\partial\bar\partial S_j\geq 0 \hspace{3ex} \mbox{(strongly)}\end{eqnarray*} for all $j\in\N$ and such that $T_j$ converges in the weak topology of currents to some $T$ as $j\to\infty$, there exists a real $(m-1,\,m-1)$-current $S$ on $X$ such that \begin{eqnarray*}T\wedge\omega_{m-1} + i\partial\bar\partial S\geq 0 \hspace{3ex} \mbox{(strongly)}.\end{eqnarray*}

\end{Lem}

  \noindent {\it Proof.} The first integral below represents the $\gamma^{n-m}$-mass of the strongly semi-positive $(m,\,m)$-current $T_j\wedge\omega_{m-1} + i\partial\bar\partial S_j$: \begin{eqnarray*}0\leq\int\limits_X(T_j\wedge\omega_{m-1} + i\partial\bar\partial S_j)\wedge\gamma^{n-m} = \int\limits_X T_j\wedge\omega_{m-1}\wedge\gamma^{n-m}\longrightarrow\int\limits_X T\wedge\omega_{m-1}\wedge\gamma^{n-m} \hspace{2ex}\mbox{as}\hspace{1ex}j\to\infty,\end{eqnarray*} where the equality followed from the Stokes theorem and the hypothesis $\partial\bar\partial\gamma^{n-m} = 0$, while the convergence of the integrals followed from the weak convergence $T_j\longrightarrow T$.

  We infer that the sequence $(T_j\wedge\omega_{m-1} + i\partial\bar\partial S_j)_{j\in\N}$ of strongly semi-positive $(m,\,m)$-currents is uniformly bounded in mass , hence there exists a weakly convergent subsequence: \begin{eqnarray*}T_{j_\nu}\wedge\omega_{m-1} + i\partial\bar\partial S_{j_\nu}\longrightarrow\widetilde{T}  \hspace{5ex}\mbox{as}\hspace{1ex}\nu\to\infty,\end{eqnarray*} with a limiting  $(m,\,m)$-current $\widetilde{T}$ that is necessarily strongly semi-positive on $X$.

  The weak convergence $T_{j_\nu}\longrightarrow T$ implies the weak convergence $T_{j_\nu}\wedge\omega_{m-1}\longrightarrow T\wedge\omega_{m-1}$, hence also the weak convergence \begin{eqnarray*}i\partial\bar\partial S_{j_\nu}\longrightarrow\widetilde{T} - T\wedge\omega_{m-1}\end{eqnarray*} as $\nu\to\infty$. This implies that $\widetilde{T} - T\wedge\omega_{m-1}\in\mbox{Im}\,(\partial\bar\partial)$. Indeed, by the current-form duality, a current is $\partial\bar\partial$-exact if and only if it vanishes on all the $\partial\bar\partial$-closed $C^\infty$ forms of the complementary bidegree. In our case, for every $\Gamma\in C^\infty_{n-m,\,n-m}(X,\,\R)$ such that $\partial\bar\partial\Gamma = 0$, we have  \begin{eqnarray*}0 = \int\limits_X i\partial\bar\partial S_{j_\nu}\wedge\Gamma\longrightarrow\int\limits_X(\widetilde{T} - T\wedge\omega_{m-1})\wedge\Gamma \hspace{5ex} \mbox{as}\hspace{1ex} \nu\to\infty.\end{eqnarray*} This implies that $\int_X(\widetilde{T} - T\wedge\omega_{m-1})\wedge\Gamma = 0$ for every $\Gamma\in C^\infty_{n-m,\,n-m}(X,\,\R)$ such that $\partial\bar\partial\Gamma = 0$.

  We conclude the existence of a real $(m-1,\,m-1)$-current $S$ on $X$ such that $\widetilde{T} - T\wedge\omega_{m-1} = i\partial\bar\partial S$. Thus: \begin{eqnarray*}T\wedge\omega_{m-1} + i\partial\bar\partial S = \widetilde{T}\geq 0 \hspace{3ex} \mbox{(strongly) on} \hspace{1ex} X,\end{eqnarray*} as claimed. \hfill $\Box$

\vspace{2ex}

We can now give the following generalisation of one half of (i) of Proposition 6.1. in [Dem92].

\begin{Prop}\label{Prop:E_omega-m_closed} Let $(X,\,\omega)$ be a compact K\"ahler $n$-dimensional manifold and let $m\in\{1,\dots , n\}$.

    Then, the convex cone ${\cal E}_{[\omega],\,m}(X)$ is {\bf closed} in $H^{1,\,1}_{BC}(X,\,\R)$.

\end{Prop}

\noindent {\it Proof.} Let ${\cal E}_{[\omega],\,m}(X)\ni\mathfrak{c}_j\underset{j\to\infty}{\longrightarrow}\mathfrak{c}\in H^{1,\,1}_{BC}(X,\,\R)$ be a convergent sequence of cohomology classes. We have to show that $\mathfrak{c}\in{\cal E}_{[\omega],\,m}(X)$.

Let $\alpha_j\in\mathfrak{c}_j$ (for every $j$) and $\alpha\in\mathfrak{c}$ be the $\Delta_{BC}$-harmonic representatives of these classes, where $\Delta_{BC}$ is the Bott-Chern Laplacian induced by the metric $\omega$. In particular, $d\alpha_j = d\alpha = 0$ for all $j$. Then, the following convergence holds in the $C^\infty$ topology: \begin{eqnarray*}\alpha_j\longrightarrow\alpha \hspace{5ex}\mbox{as}\hspace{1ex}j\to\infty.\end{eqnarray*}

Now, for every $j$, since $\mathfrak{c}_j = [\alpha_j]_{BC}\in{\cal E}_{[\omega],\,m}(X)$, there exists a real $(m-1,\,m-1)$-current $S_j$ on $X$ such that \begin{eqnarray*}\alpha_j\wedge\omega_{m-1} + i\partial\bar\partial S_j\geq 0 \hspace{3ex} \mbox{(strongly)}\end{eqnarray*} as an $(m,\,m)$-current on $X$.

Thus, the hypotheses of Lemma \ref{Lem:limiting-S-current} are satisfied with $T_j=\alpha_j$ (because convergence in the $C^\infty$ topology implies convergence in the weak topology of currents for $\alpha_j\underset{j\to\infty}{\longrightarrow}\alpha$) and $\gamma=\omega$. That lemma ensures the existence of a real $(m-1,\,m-1)$-current $S$ on $X$ such that \begin{eqnarray*}\alpha\wedge\omega_{m-1} + i\partial\bar\partial S\geq 0 \hspace{3ex} \mbox{(strongly)}\end{eqnarray*} as an $(m,\,m)$-current on $X$. This, in turn, ensures that $\mathfrak{c}=[\alpha]_{BC}\in{\cal E}_{[\omega],\,m}(X)$, as desired.  \hfill $\Box$

\subsection{An equivalence between two $m$-positivity notions}\label{subsection:m-pos_equiv} For reasons that will become apparent further afield, we start by introducing the following Laplace-type differential operator acting on $\R$-valued or $\C$-valued $C^\infty$ functions on $X$ and depending on a given pair $(\omega,\,\Omega)$ of suitable differential forms.

\begin{Def}\label{Def:P_omega_Omega} Let $X$ be a compact complex $n$-dimensional manifold and let $m\in\{1,\dots , n\}$. Suppose there exist a K\"ahler metric $\omega$ on $X$ and a real form $\Omega\in C^\infty_{n-m,\,n-m}(X,\,\R)$ such that \begin{eqnarray}\label{eqn:hypotheses_Omega}\Omega>0 \hspace{2ex} \mbox{(weakly) \hspace{3ex} and \hspace{3ex}} \partial\bar\partial\Omega = 0.\end{eqnarray}

We set: \begin{eqnarray*}   P = P_{\omega,\,\Omega}:C^\infty(X,\,\R)\longrightarrow C^\infty(X,\,\R), \hspace{5ex} P_{\omega,\,\Omega}(\varphi):=-\frac{i\partial\bar\partial\varphi\wedge\omega^{m-1}\wedge\Omega}{dV_\omega}.\end{eqnarray*}

The same definition is made for $P = P_{\omega,\,\Omega}:C^\infty(X,\,\C)\longrightarrow C^\infty(X,\,\C)$.
  
\end{Def}

The adjoint operator is computed in

\begin{Lem}\label{Lem:P-operator_adjoint} Let $P^\star_{\omega,\,\Omega}:C^\infty(X,\,\C)\longrightarrow C^\infty(X,\,\C)$ be the $L^2_\omega$-adjoint of the operator introduced in Definition \ref{Def:P_omega_Omega}. For every $C^\infty$ function $\varphi:X\longrightarrow\C$, the following identity holds when $X$ is compact: \begin{eqnarray*}\bigg(P^\star_{\omega,\,\Omega} - P_{\omega,\,\Omega}\bigg)\,(\varphi) = \frac{i(\bar\partial\varphi\wedge\partial\Omega - \partial\varphi\wedge\bar\partial\Omega)\wedge\omega^{m-1}}{dV_\omega}.\end{eqnarray*}

In particular, $P^\star_{\omega,\,\Omega}$ and $P_{\omega,\,\Omega}$ differ by a first-order operator, so $P^\star_{\omega,\,\Omega}$ is elliptic of order two with no zero$^{th}$-order terms.

\end{Lem}

\noindent {\it Proof.} Let $\varphi,\psi:X\longrightarrow\C$ be arbitrary $C^\infty$ functions. We have: \begin{eqnarray*}\langle\langle P_{\omega,\,\Omega}(\varphi),\,\psi\rangle\rangle = -\int\limits_X\overline\psi\, i\partial\bar\partial\varphi\wedge\omega^{m-1}\wedge\Omega  \hspace{3ex}\mbox{and}\hspace{3ex} \langle\langle \varphi,\,P_{\omega,\,\Omega}(\psi)\rangle\rangle = -\int\limits_X\varphi\, i\partial\bar\partial\,\overline\psi\wedge\omega^{m-1}\wedge\Omega.\end{eqnarray*}

On the other hand, we have: \begin{eqnarray*}i\partial\bar\partial(\varphi\,\overline\psi) = \varphi\,i\partial\bar\partial\,\overline\psi + i\partial\varphi\wedge\bar\partial\,\overline\psi + \overline\psi\,i\partial\bar\partial\varphi + i\partial\overline\psi\wedge\bar\partial\varphi.\end{eqnarray*}

Plugging the value for $\varphi\,i\partial\bar\partial\,\overline\psi$ given by this equality into one of the previous expressions, we get the second equality below: \begin{eqnarray}\label{eqn:P-operator-adjoint_proof_1}\nonumber\langle\langle P^\star_{\omega,\,\Omega}(\varphi),\,\psi\rangle\rangle = \langle\langle \varphi,\,P_{\omega,\,\Omega}(\psi)\rangle\rangle & = & \int\limits_X \overline\psi\,i\partial\bar\partial\varphi\wedge\omega^{m-1}\wedge\Omega + \int\limits_X i\partial\varphi\wedge\bar\partial\,\overline\psi\wedge\omega^{m-1}\wedge\Omega \\
  & + & \int\limits_X i\partial\overline\psi\wedge\bar\partial\varphi\wedge\omega^{m-1}\wedge\Omega - \int\limits_X i\partial\bar\partial(\varphi\,\overline\psi)\wedge\omega^{m-1}\wedge\Omega.\end{eqnarray}

Now, thanks to $\omega^{m-1}$ being $\partial$-closed and $\bar\partial$-closed and to $\partial\bar\partial\Omega = 0$, two applications of the Stokes theorem yield the following result for the last integral: \begin{eqnarray*}\int\limits_X i\partial\bar\partial(\varphi\,\overline\psi)\wedge\omega^{m-1}\wedge\Omega & = & i\, \int\limits_X\partial\bigg(\bar\partial(\varphi\,\overline\psi)\wedge\omega^{m-1}\wedge\Omega\bigg) + i\, \int\limits_X\bar\partial(\varphi\,\overline\psi)\wedge\omega^{m-1}\wedge\partial\Omega \\
  & = & i\, \int\limits_X\bar\partial\bigg(\varphi\,\overline\psi\,\omega^{m-1}\wedge\partial\Omega\bigg) + i\, \int\limits_X\varphi\,\overline\psi\,\omega^{m-1}\wedge\partial\bar\partial\Omega = 0.\end{eqnarray*}

Similarly, for the last integral on the first line of (\ref{eqn:P-operator-adjoint_proof_1}), we get: \begin{eqnarray*}\int\limits_X i\partial\varphi\wedge\bar\partial\,\overline\psi\wedge\omega^{m-1}\wedge\Omega = -\int\limits_X\overline\psi\, i\partial\bar\partial\varphi\wedge\omega^{m-1}\wedge\Omega -i\,\int\limits_X\overline\psi\,\partial\varphi\wedge\omega^{m-1}\wedge\bar\partial\Omega,\end{eqnarray*} (so, this computes the sum of the last two integrals on the first line of (\ref{eqn:P-operator-adjoint_proof_1})), while for the first integral on the second line of (\ref{eqn:P-operator-adjoint_proof_1}), we get: \begin{eqnarray*}\int\limits_X i\partial\overline\psi\wedge\bar\partial\varphi\wedge\omega^{m-1}\wedge\Omega & = & -\int\limits_X\overline\psi\, i\partial\bar\partial\varphi\wedge\omega^{m-1}\wedge\Omega + i\,\int\limits_X\overline\psi\,\bar\partial\varphi\wedge\omega^{m-1}\wedge\partial\Omega \\
  & = & \langle\langle P_{\omega,\,\Omega}(\varphi),\,\psi\rangle\rangle + i\,\int\limits_X\overline\psi\,\bar\partial\varphi\wedge\omega^{m-1}\wedge\partial\Omega.\end{eqnarray*}

All these computation results transform (\ref{eqn:P-operator-adjoint_proof_1}) into: \begin{eqnarray*}\langle\langle P^\star_{\omega,\,\Omega}(\varphi),\,\psi\rangle\rangle = \langle\langle P_{\omega,\,\Omega}(\varphi),\,\psi\rangle\rangle + \int\limits_X \overline\psi\,i(\bar\partial\varphi\wedge\partial\Omega - \partial\varphi\wedge\bar\partial\Omega)\wedge\omega^{m-1}.\end{eqnarray*} This equality holds for all $C^\infty$ functions $\varphi,\psi:X\longrightarrow\C$. This proves the contention. \hfill $\Box$

\begin{Cor}\label{Cor:2-space-decomposition} If the manifold $X$ is compact, $\ker\,\bigg(P_{\omega,\,\Omega}:C^\infty(X,\,\R)\longrightarrow C^\infty(X,\,\R\bigg) = \R$ and $\ker\,\bigg(P_{\omega,\,\Omega}:C^\infty(X,\,\C)\longrightarrow C^\infty(X,\,\C)\bigg) =\C$. Moreover, the $L^2_\omega$-orthogonal decompositions hold: \begin{eqnarray}\label{eqn:2-space-decomposition}C^\infty(X,\,\R) = \R\oplus\mbox{Im}\,P_{\omega,\,\Omega} \hspace{3ex}\mbox{and}\hspace{3ex} C^\infty(X,\,\C) = \C\oplus\mbox{Im}\,P_{\omega,\,\Omega}.\end{eqnarray}

\end{Cor}  

\noindent {\it Proof.} The statement follows from Lemma \ref{Lem:P-operator_adjoint} and the maximum principle applied to $P^\star_{\omega,\,\Omega}$ on the compact $X$. \hfill $\Box$

\vspace{2ex}

With this preparation under our belt, we will now prove Theorem \ref {The:m-pos_equiv_introd} which is the transcendental version of the following

\begin{The}\label{The:m-pos_equiv} Let $(X,\omega)$ be a compact K\"ahler manifold with $\mbox{dim}_\C X = n$. Fix $m\in\{1,\dots , n\}$.







\vspace{1ex}

 For every holomorphic line bundle $L$ over $X$, the following two statements are equivalent:

  \vspace{1ex}

 (a) There exist a $C^\infty$ fibre metric $h$ on $L$ and a form $\Omega_0\in C^\infty_{n-m,\,n-m}(X,\,\R)$ satisfying (\ref{eqn:m-pos_equiv_Omega-prop}) such that $i\Theta_h(L)\wedge\omega^{m-1}\wedge\Omega_0 >0$ on $X$;

 \vspace{1ex}

 (b) The dual line bundle $L^{-1}$ is {\bf not $[\omega]$-$m$-pseudo-effective}.

\end{The}

\begin{proof} Since Theorem \ref{The:m-pos_equiv} is the special case of Theorem \ref{The:m-pos_equiv_introd} for integral cohomology classes $\mathfrak{c}$, it suffices to prove the latter.

  \vspace{1ex}

  $(a)\implies(b)$ Suppose there exist $\alpha\in\mathfrak{c}$ and $\Omega_0\in C^\infty_{n-m,\,n-m}(X,\,\R)$ with the properties under (a). Since the $(n,\,n)$-form $\alpha\wedge\omega^{m-1}\wedge\Omega_0$ is positive at every point, we have: \begin{eqnarray*}\int\limits_X\alpha\wedge\omega^{m-1}\wedge\Omega_0 >0.\end{eqnarray*}

  Reasoning by contradiction, suppose that (b) does not hold. Then, the class $-\mathfrak{c}$ is $[\omega]$-$m$-pseudo-effective. By the duality Lemma \ref{Lem:duality_generalised-Lamari} (reworded as $(3)$ of Proposition and Definition \ref{Prop-Def:F_m-G_m_cones}), this is equivalent to \begin{eqnarray*}-\int\limits_X\alpha\wedge\omega^{m-1}\wedge\Omega \geq 0\end{eqnarray*} for every $\Omega\in C_{n-m,\,n-m}^\infty(X,\,\R)$ satisfying properties (\ref{eqn:m-pos_equiv_Omega-prop}). Taking $\Omega = \Omega_0$, we get a contradiction.

\vspace{1ex}

$(b)\implies(a)$ Suppose that $-\mathfrak{c}\in H^{1,\,1}_{BC}(X,\,\R)$ is not $[\omega]$-$m$-pseudo-effective. By the duality Lemma \ref{Lem:duality_generalised-Lamari} (reworded as $(3)$ of Proposition and Definition \ref{Prop-Def:F_m-G_m_cones}), this is equivalent to there existing a form $\Omega_0\in C_{n-m,\,n-m}^\infty(X,\,\R)$ satisfying properties (\ref{eqn:m-pos_equiv_Omega-prop}) such that \begin{eqnarray*}-\int\limits_X\alpha\wedge\omega^{m-1}\wedge\Omega_0 < 0\end{eqnarray*} for some (hence any) $C^\infty$ $d$-closed real $(1,\,1)$-form $\alpha\in\mathfrak{c}$. (Note that a change of representative of the class $\mathfrak{c}$ from $\alpha$ to $\alpha + \partial\bar\partial\zeta$ would replace $\alpha\wedge\omega^{m-1}$ by $\alpha\wedge\omega^{m-1} + \partial\bar\partial(\zeta\wedge\omega^{m-1})$, since $\omega$ is K\"ahler, and the above integral would remain unchanged thanks to the Stokes theorem and to $\partial\bar\partial\Omega_0=0$.)

  Fix an arbitrary $C^\infty$ $d$-closed real $(1,\,1)$-form $\alpha_0\in\mathfrak{c}$ and consider the positive constant \begin{eqnarray*}I_0:=\int\limits_X\alpha_0\wedge\omega^{m-1}\wedge\Omega_0 > 0.\end{eqnarray*}

  To prove (a), we will prove the existence of a $C^\infty$ function $f:X\longrightarrow\R$ such that the form $\alpha:=\alpha_0 + i\partial\bar\partial f\in\mathfrak{c}$ satisfies the condition $\alpha\wedge\omega^{m-1}\wedge\Omega_0 >0$ everywhere on $X$. This condition is equivalent to \begin{eqnarray*}\frac{\alpha_0\wedge\omega^{m-1}\wedge\Omega_0}{dV_\omega} + \frac{i\partial\bar\partial f\wedge\omega^{m-1}\wedge\Omega_0}{dV_\omega} > 0,\end{eqnarray*} namely to the $C^\infty$ function $g_{\alpha_0,\,\omega,\,\Omega_0} - P_{\omega,\,\Omega_0}(f)$ being positive at every point of $X$, where we have set $g_{\alpha_0,\,\omega,\,\Omega_0}:=(\alpha_0\wedge\omega^{m-1}\wedge\Omega_0)/dV_\omega$, an $\R$-valued $C^\infty$ function on $X$, while $P_{\omega,\,\Omega_0}:C^\infty(X,\,\R)\longrightarrow C^\infty(X,\,\R)$ is the operator introduced in Definition \ref{Def:P_omega_Omega}.

  For this to happen, it suffices to prove the existence of a $C^\infty$ function $f:X\longrightarrow\R$ such that the $C^\infty$ function $g_{\alpha_0,\,\omega,\,\Omega_0} - P_{\omega,\,\Omega_0}(f)$ is constant, equal to the positive constant $I_0/\mbox{Vol}_\omega(X)$, where $\mbox{Vol}_\omega(X):=\int_XdV_\omega$. Thus, it suffices to prove the existence of a $C^\infty$ function $f:X\longrightarrow\R$ such that \begin{eqnarray*}P_{\omega,\,\Omega_0}(f) = g_{\alpha_0,\,\omega,\,\Omega_0} - \frac{I_0}{\mbox{Vol}_\omega(X)}   \hspace{5ex}\mbox{everywhere on}\hspace{1ex} X.\end{eqnarray*}

  Such a function $f$ exists if and only if the function $g_{\alpha_0,\,\omega,\,\Omega_0} - I_0/\mbox{Vol}_\omega(X)$ lies in the image of $P_{\omega,\,\Omega_0}$. This is further equivalent, thanks to the first two-space $L^2_\omega$-orthogonal decomposition (\ref{eqn:2-space-decomposition}), to the function $g_{\alpha_0,\,\omega,\,\Omega_0} - I_0/\mbox{Vol}_\omega(X)$ being $L^2_\omega$-orthogonal to the real constants, namely to the first equality in the following sequence of equivalences: \begin{eqnarray*}\int\limits_X\bigg(g_{\alpha_0,\,\omega,\,\Omega_0} - \frac{I_0}{\mbox{Vol}_\omega(X)}\bigg)\,dV_\omega = 0 \iff \int\limits_X\alpha_0\wedge\omega^{m-1}\wedge\Omega_0 - \frac{I_0}{\mbox{Vol}_\omega(X)}\,\int\limits_XdV_\omega = 0 \iff I_0 - I_0 = 0,\end{eqnarray*} where the last equality, which is obviously true, followed from the definition of $I_0$.

\end{proof}

\section{A Monge-Amp\`ere-type PDE for positive degree solutions}\label{section:equation} In this section, we study the equation of Definition \ref{Def:equation} and prove the uniqueness Theorem \ref{The:equation_uniqueness_introd}.

A word about the notation that will be used throughout the rest of the paper: for any Hermitian metric $\gamma$ on $X$, $\star_\gamma:\Lambda^{p,\,q}T^\star X\longrightarrow\Lambda^{n-q,\,n-p}T^\star X$ denotes the associated Hodge star operator defined, in any bidegree $(p,\,q)$ and at every point of $X$, by the requirement: \begin{eqnarray*}u\wedge\star_\gamma\bar{v} = \langle u,\,v\rangle_\gamma\,dV_\gamma
\end{eqnarray*} for all $(p,\,q)$-forms $u,v$, where $dV_\gamma :=\gamma^n/n!$ is the volume form and $\langle\,\cdot\,,\,\cdot\,\rangle_\gamma$ is the pointwise inner product induced by $\gamma$. We also let, for every $p\in\{1,\dots , n\}$: \begin{eqnarray*}\gamma_p:=\frac{\gamma^p}{p!}.\end{eqnarray*}

  Note that in (\ref{eqn:equation}) we have: $(\alpha + i\partial\bar\partial u)\wedge\omega_{n-m-1}\in C^\infty_{n-1,\,n-1}(X,\,\R)$, hence \begin{eqnarray*}\star_\omega\bigg((\alpha + i\partial\bar\partial u)\wedge\omega_{n-m-1}\bigg)\in C^\infty_{1,\,1}(X,\,\R),\end{eqnarray*} so the $n^{th}$ power of this $(1,\,1)$-form is a volume form on $X$ that is prescribed by $dV$ under (\ref{eqn:equation}).

\vspace{1ex}

  Before going further, we observe the following fact that will come in handy.

  \begin{Lem}\label{Lem:Laplacians_link} Let $\rho>0$ be a Hermitian metric on a complex manifold $X$ with $\mbox{dim}_\C X = n$. Then, for all $p,q\in\{0,\dots , n\}$ and every $u\in C^\infty_{p,\,q}(X,\,\C)$ such that $\bar\partial_\rho^\star u =0$, the following formula holds:  \begin{eqnarray}\label{eqn:Laplacians_link_Hermitian}\Lambda_\rho(i\partial\bar\partial u) = i\partial\bar\partial\Lambda_\rho u - \Delta''_\rho u + \partial\partial_\rho^\star u + \partial\tau_\rho^\star u - \bar\tau_\rho^\star\bar\partial u,\end{eqnarray} where $\Delta''_\rho: = \bar\partial\bar\partial_\tau^\star + \bar\partial_\tau^\star\bar\partial$ is the $\bar\partial$-Laplacian induced by $\rho$ and $\tau_\rho:= [\Lambda_\rho,\,\partial\rho\wedge\cdot\,]$ is the torsion operator associated with $\rho$.

    In particular, for every {\bf K\"ahler} metric $\omega$ (if any) on $X$ and every $u\in C^\infty_{p,\,q}(X,\,\C)$ such that $u\in\ker\partial^\star_\omega\cap\ker\bar\partial^\star_\omega$, the following formulae hold:  \begin{eqnarray}\label{eqn:Laplacians_link_Kaehler}\Lambda_\omega(i\partial\bar\partial u) & = & i\partial\bar\partial\Lambda_\omega u - \Delta''_\omega u \\
 \label{eqn:Lambda-l-iddbar-u} \Lambda_\omega^l(i\partial\bar\partial u) & = & i\partial\bar\partial\Lambda_\omega^l u - l\,\Delta''_\omega\Lambda_\omega^{l-1} u, \hspace{6ex} l=1,\dots , m.  \end{eqnarray}

\end{Lem}

  \noindent {\it Proof.} Recall the following standard commutation relations ([Dem84], see also [Dem97, VII, $\S.1$]) that hold for any Hermitian metric $\rho$ on any complex manifold: \begin{eqnarray}\label{eqn:standard-comm-rel} (i)\,\,[\Lambda_\rho,\,\bar\partial] = -i(\partial^\star_\rho + \tau^\star_\rho);  \hspace{3ex} (ii)\,\,[\Lambda_\rho,\,\partial] = i(\bar\partial^\star_\rho + \bar\tau^\star_\rho).\end{eqnarray}

  Using first (ii) to commute (with an error term) $\Lambda_\rho$ with $\partial$ and then (i) to commute (with an error term) $\Lambda_\rho$ with $\bar\partial$, we successively get for every $u\in C^\infty_{p,\,q}(X,\,\C)$ such that $\bar\partial_\rho^\star u =0$: \begin{eqnarray*}\Lambda_\rho(\partial\bar\partial u) = \partial\Lambda_\rho(\bar\partial u) + i\bar\partial^\star_\rho\bar\partial u  + i\bar\tau^\star_\rho\bar\partial u = \partial\bar\partial\Lambda_\rho u -i\partial\partial^\star_\rho u -i\partial\tau^\star_\rho u + i\,\Delta''_\rho u + i\bar\tau^\star_\rho\bar\partial u.\end{eqnarray*} This implies (\ref{eqn:Laplacians_link_Hermitian}) after multiplication by $i$. For the second equality above, we used the hypothesis $\bar\partial_\rho^\star u =0$ to get: $\Delta''_\rho u = \bar\partial\bar\partial^\star_\rho u + \bar\partial^\star_\rho\bar\partial u = \bar\partial^\star_\rho\bar\partial u$.

  When $\rho=\omega$ is a K\"ahler metric, the torsion operator $\tau_\omega$ vanishes identically, hence so do its adjoint and its conjugate. If, moreover, we assume that $\partial^\star_\rho u = 0$, (\ref{eqn:Laplacians_link_Hermitian}) becomes (\ref{eqn:Laplacians_link_Kaehler}).

  Finally, note that (\ref{eqn:Lambda-l-iddbar-u}) for $l=1$ is precisely (\ref{eqn:Laplacians_link_Kaehler}). We will prove that it holds for every $l$ by induction. Supposing that (\ref{eqn:Lambda-l-iddbar-u}) holds for $l$ and taking $\Lambda_\omega$, we get the first equality below: \begin{eqnarray*}\Lambda_\omega^{l+1}(i\partial\bar\partial u) & = & \Lambda_\omega(i\partial\bar\partial\Lambda_\omega^l u) - l\,\Lambda_\omega(\Delta''_\omega\Lambda_\omega^{l-1} u) = i\partial\bar\partial\Lambda_\omega^{l+1} u - \Delta''_\omega\Lambda_\omega^l u - l\,\Delta''_\omega\Lambda_\omega^l u \\
  & = & i\partial\bar\partial\Lambda_\omega^{l+1} u - (l+1)\, \Delta''_\omega\Lambda_\omega^l u,\end{eqnarray*} where the second equality followed by applying (\ref{eqn:Laplacians_link_Kaehler}) with $\Lambda_\omega^l u$ in place of $u$ and by commuting $\Lambda_\omega$ with $\Delta''_\omega$. (This commutation holds because $\omega$ is K\"ahler. Indeed, the K\"ahler commutation relations imply that $L_\omega$ and $\Delta''_\omega$ commute, hence so do their adjoints, while $\Delta''_\omega$ is self-adjoint.) Note that $\partial^\star_\omega(\Lambda_\omega^l u) = \Lambda_\omega^l(\partial^\star_\omega u) = 0$ and $\bar\partial^\star_\omega(\Lambda_\omega^l u) = \Lambda_\omega^l(\bar\partial^\star_\omega u) = 0$ (hence (\ref{eqn:Laplacians_link_Kaehler}) applies in the above situation) since $\partial$ and $\bar\partial$ commute with $L_\omega$ (hence $\partial^\star_\omega$ and $\bar\partial^\star_\omega$ commute with $\Lambda_\omega$) thanks to $\omega$ being K\"ahler.   \hfill $\Box$

\vspace{2ex}

On the other hand, recall that any differential form has a (uniquely determined) {\it Lefschetz decomposition} with respect to any given Hermitian metric $\omega$. In our case, we will use this for forms of bidegree $(r,\,r)$ with $r$ arbitrary. Explicitly, on any compact $n$-dimensional Hermitian manifold $(X,\,\omega)$, for every form $\zeta\in C^\infty_{r,\,r}(X,\,\C)$ there exist unique {\it primitive} (with respect to $\omega$, also called {\it $\omega$-primitive} to avoid confusion) forms $\zeta_{prim}^{(l)}\in C^\infty_{r-l,\,r-l}(X,\,\C)$ for $l=0,\dots , r$ such that \begin{eqnarray}\label{eqn:Lefschetz-decomp}\zeta = \zeta_{prim}^{(0)} + \zeta_{prim}^{(1)}\wedge\omega + \dots + \zeta_{prim}^{(l)}\wedge\omega^l + \dots + \zeta_{prim}^{(r)}\,\omega^r = \sum\limits_{l=0}^r L_\omega^l(\zeta_{prim}^{(l)}),\end{eqnarray} where $L_\omega:\Lambda^{p,\,q}T^\star X\longrightarrow\Lambda^{p+1,\,q+1}T^\star X$ is the Lefschetz operator of multiplication by $\omega$. Recall that a $(p,\,q)$-form $v$ is said to be $\omega$-primitive if $p+q\leq n$ and $v\wedge\omega^{n-p-q+1} = 0$. This condition is equivalent to $p+q\leq n$ and $\Lambda_\omega v = 0$. The Lefschetz decomposition (\ref{eqn:Lefschetz-decomp}) shows that every $\zeta\in C^\infty_{r,\,r}(X,\,\C)$ is completely determined by its primitive components (that can be called the {\it $\omega$-primitive coordinates}), so we can write $\zeta=(\zeta_{prim}^{(l)})_{0\leq l\leq r}$.

Another preliminary observation that will be needed is the following

\begin{Lem}\label{Lem:star_product_zeta_omega-powers} Let $(X,\,\omega)$ be a complex Hermitian manifold with $\mbox{dim}_\C X = n$.

\vspace{1ex}

$(1)$\, Let $r\in\{1,\dots , n\}$ be such that $2r\leq n$. Fix an arbitrary $\zeta\in C^\infty_{r,\,r}(X,\,\C)$. Considering the Lefschetz decomposition (\ref{eqn:Lefschetz-decomp}) of $\zeta$, we have: \begin{eqnarray}\label{eqn:star_product_zeta_omega-powers}\star_\omega\bigg(\zeta\wedge\omega_{n-r-1}\bigg) = -\frac{(n-2)!}{(n-r-1)!}\,\zeta_{prim}^{(r-1)} + \frac{(n-1)!}{(n-r-1)!}\,\zeta_{prim}^{(r)}\,\omega,\end{eqnarray} where the forms on the right are given by the formulae: \begin{eqnarray}\label{eqn:star_Lefschetz_computed-components_r-1}\zeta_{prim}^{(r-1)} & = & \frac{(n-r-1)!}{(n-2)!\,(r-1)!}\,\bigg(\Lambda_\omega^{r-1}\zeta - \frac{1}{n}\,(\Lambda_\omega^r\zeta)\,\omega\bigg) \\
 \label{eqn:star_Lefschetz_computed-components_r}   \zeta_{prim}^{(r)} & = & \frac{(n-r)!}{n!\,r!}\,\Lambda_\omega^r\zeta.\end{eqnarray}

This transforms (\ref{eqn:star_product_zeta_omega-powers}) to \begin{eqnarray}\label{eqn:star_product_zeta_omega-powers_final}\star_\omega\bigg(\zeta\wedge\omega_{n-r-1}\bigg) = \frac{1}{(r-1)!}\,\bigg(-\Lambda_\omega^{r-1}\zeta + \frac{1}{r}\,(\Lambda_\omega^r\zeta)\,\omega\bigg).\end{eqnarray}

\vspace{1ex}

$(2)$\, (Standard) Suppose that the metric $\omega$ is {\bf K\"ahler}. Then, for every bidegree $(p,\,q)$ such that $p+q\leq n$ and every $v\in C^2_{p,\,q}(X,\,\C)$, the following implication holds: \begin{eqnarray}\label{eqn:primitive_Delta''_implication}v \hspace{2ex}\mbox{is}\hspace{1ex}\omega\mbox{-primitive} \implies \Delta''_\omega v \hspace{2ex}\mbox{is}\hspace{1ex}\omega\mbox{-primitive}.\end{eqnarray}

In particular, if $v = \sum_{l\geq 0}v_{prim}^{(l)}\wedge\omega^l$ is the Lefschetz decomposition of $v$, then $\Delta''_\omega v = \sum_{l\geq 0}\Delta''_\omega v_{prim}^{(l)}\wedge\omega^l$ is the Lefschetz decomposition of $\Delta''_\omega v$. Moreover, the following equivalence holds: \begin{eqnarray}\label{eqn:harmonicity_form_prim-coordinates_equivalence}\Delta''_\omega v = 0 \iff \Delta''_\omega v_{prim}^{(l)} = 0 \hspace{2ex}\mbox{for all}\hspace{1ex} l.\end{eqnarray}

\vspace{1ex}

$(3)$\,  Suppose that the metric $\omega$ is {\bf K\"ahler}. Then, for every $m\in\{1,\dots , n-1\}$ and every $\alpha\in C^\infty_{m,\,m}(X,\,\C)$, the following implication holds: \begin{eqnarray}\label{eqn:harmonicity_alpha_alpha-wedge-omega-power_implication}\Delta''_\omega\alpha = 0  \implies \Delta''_\omega\bigg(\star_\omega(\alpha\wedge\omega_{n-m-1})\bigg) = 0.\end{eqnarray} In particular, $\star_\omega(\alpha\wedge\omega_{n-m-1})\in\ker d$ whenever $\alpha\in\ker\Delta''_\omega$.

\end{Lem}

\noindent {\it Proof.} $(1)$\, Multiplying (\ref{eqn:Lefschetz-decomp}) by $\omega_{n-r-1}$, we get: \begin{eqnarray*}\zeta\wedge\omega_{n-r-1} = \sum\limits_{l=0}^r\zeta_{prim}^{(l)}\wedge\frac{\omega^{n-r+l-1}}{(n-r-1)!} = \zeta_{prim}^{(r-1)}\wedge\frac{\omega^{n-2}}{(n-r-1)!} + \zeta_{prim}^{(r)}\,\frac{\omega^{n-1}}{(n-r-1)!}.\end{eqnarray*} The reason why only the terms corresponding to $l=r-1$ and $l=r$ from the first sum survive into the second is that the total degree of each $\zeta_{prim}^{(l)}$ is $2r-2l$, so the primitivity of this form implies that $\zeta_{prim}^{(l)}\wedge\omega^{n-r+l-1} = 0$ as soon as $n-(2r-2l)<n-r+l-1$. This last inequality amounts to $l<r-1$.

Taking $\star_\omega$, we get: \begin{eqnarray*}\star_\omega\bigg(\zeta\wedge\omega_{n-r-1}\bigg) & = & \frac{(n-2)!}{(n-r-1)!}\,\star_\omega\bigg(\zeta_{prim}^{(r-1)}\wedge\omega_{n-2}\bigg) + \frac{(n-1)!}{(n-r-1)!}\,\star_\omega\bigg(\zeta_{prim}^{(r)}\,\omega_{n-1}\bigg) \\
  & = & -\frac{(n-2)!}{(n-r-1)!}\,\zeta_{prim}^{(r-1)} + \frac{(n-1)!}{(n-r-1)!}\, \zeta_{prim}^{(r)}\,\omega,\end{eqnarray*} where for the last equality (which is precisely (\ref{eqn:star_product_zeta_omega-powers})) we used the following standard formula (cf. e.g. [Voi02, Proposition 6.29, p. 150]) giving the image under $\star_\omega$ of any {\it $\omega$-primitive} form $v$ of any bidegree $(p, \, q)$: \begin{eqnarray}\label{eqn:prim-form-star-formula-gen}\star_\omega\, v = (-1)^{k(k+1)/2}\, i^{p-q}\, \frac{\omega^{n-p-q}\wedge v}{(n-p-q)!}, \hspace{2ex} \mbox{where}\,\, k:=p+q.\end{eqnarray} In our case, we applied this formula to the $(1,\,1)$-form $v:=\zeta_{prim}^{(r-1)}$ and then used the fact that $\star_\omega\star_\omega$ is the identity operator on forms of even degree. Meanwhile, $\zeta_{prim}^{(r)}$ is of bidegree $(0,\,0)$ (i.e. a function), so it commutes with $\star_\omega$, while $\star_\omega\omega_{n-1} = \omega$.

\vspace{1ex}

In order to compute $\zeta_{prim}^{(r-1)}$ and $\zeta_{prim}^{(r)}$ and prove formulae (\ref{eqn:star_Lefschetz_computed-components_r-1}) and (\ref{eqn:star_Lefschetz_computed-components_r}), we will use the following standard commutation formula (cf. [Dem97, Chapter VI, $\S5.2.$]) that holds in every degree $k$ on every $n$-dimensional complex manifold and for every integer $r\geq 1$: \begin{eqnarray}\label{eqn:commutation_Lr-Lambda}[L_\omega^r,\,\Lambda_\omega] = r(k-n+r-1)\,L_\omega^{r-1} \hspace{6ex}\mbox{on}\hspace{2ex} k\mbox{-forms}.\end{eqnarray}

\vspace{1ex}

Taking $\Lambda_\omega$ in (\ref{eqn:Lefschetz-decomp}), we get: \begin{eqnarray*}\Lambda_\omega\zeta = \sum\limits_{l=1}^r \bigg[\Lambda_\omega,\,L_\omega^l\bigg]\,(\zeta_{prim}^{(l)}) = -\sum\limits_{l=0}^r \bigg[L_\omega^l,\,\Lambda_\omega\bigg]\,(\zeta_{prim}^{(l)}) = -\sum\limits_{l=1}^r l(2r-l-n-1)\,L_\omega^{l-1}\,(\zeta_{prim}^{(l)}),\end{eqnarray*} where the first equality followed from $\Lambda_\omega\zeta_{prim}^{(l)} = 0$ for every $l$, while the third equality followed from (\ref{eqn:commutation_Lr-Lambda}). Taking $\Lambda_\omega$ again, we get: \begin{eqnarray*}\Lambda_\omega^2\zeta & = & -\sum\limits_{l=1}^r l(2r-l-n-1)\,\bigg[\Lambda_\omega,\,L_\omega^{l-1}\bigg]\,(\zeta_{prim}^{(l)}) = \sum\limits_{l=1}^r l(2r-l-n-1)\,\bigg[L_\omega^{l-1},\,\Lambda_\omega\bigg]\,(\zeta_{prim}^{(l)}) \\
  & = & \sum\limits_{l=2}^r l(l-1)(2r-l-n-1)(2r-l-n-2)\,L_\omega^{l-2}\,(\zeta_{prim}^{(l)}),\end{eqnarray*} where (\ref{eqn:commutation_Lr-Lambda}) was applied again to yield the last equality.

Iterating this process inductively, we get for every $s\in\{1,\dots , r\}$: \begin{eqnarray*}\Lambda_\omega^s\zeta = (-1)^s\,\sum\limits_{l=s}^r l(l-1)\cdots (l-s+1)\,(2r-l-n-1)(2r-l-n-2)\cdots (2r-l-n-s)\,L_\omega^{l-s}\,(\zeta_{prim}^{(l)}).\end{eqnarray*}

$\bullet$ In particular, when $s=r$, this spells: \begin{eqnarray*}\Lambda_\omega^r\zeta = (-1)^r\, r!\,(-1)^r\,(n+1-r)(n+2-r)\cdots(n+r-r)\,\zeta_{prim}^{(r)} = \frac{r!\,n!}{(n-r)!}\,\zeta_{prim}^{(r)}.\end{eqnarray*} This proves (\ref{eqn:star_Lefschetz_computed-components_r}).

\vspace{1ex}

$\bullet$ Meanwhile, in the particular case where $s=r-1$, the above formula for $\Lambda_\omega^s\zeta$ spells: \begin{eqnarray*}\Lambda_\omega^{r-1}\zeta & = & (-1)^{r-1}\, (r-1)(r-2)\cdots 1\,(r-n)(r-n-1)\cdots(2-n)\,\zeta_{prim}^{(r-1)} \\
  & + & (-1)^{r-1}\, r(r-1)\cdots 2\,(r-n-1)(r-n-2)\cdots(1-n)\,\zeta_{prim}^{(r)}\,\omega \\
  & = &  (r-1)!\,(n-r)(n-r+1)\cdots(n-2)\,\zeta_{prim}^{(r-1)} + r!\,(n-r+1)(n-r+2)\cdots(n-1)\,\zeta_{prim}^{(r)}\,\omega \\
 & = &  \frac{(r-1)!\,(n-2)!}{(n-r-1)!}\,\zeta_{prim}^{(r-1)} + \frac{r!\,(n-1)!}{(n-r)!}\,\zeta_{prim}^{(r)}\,\omega.\end{eqnarray*} Now, replacing $\zeta_{prim}^{(r)}$ with its value given in (\ref{eqn:star_Lefschetz_computed-components_r}) (and proved above), the last equality transforms to: \begin{eqnarray*}\Lambda_\omega^{r-1}\zeta = \frac{(r-1)!\,(n-2)!}{(n-r-1)!}\,\zeta_{prim}^{(r-1)} + \frac{1}{n}\,(\Lambda_\omega^r\zeta)\,\omega.\end{eqnarray*} Therefore, after multiplying the value for $\zeta_{prim}^{(r-1)}$ given by the last expression with $(n-2)!/(n-r-1)!$, we get: \begin{eqnarray*}\frac{(n-2)!}{(n-r-1)!}\,\zeta_{prim}^{(r-1)} = \frac{1}{(r-1)!}\,\bigg(\Lambda_\omega^{r-1}\zeta - \frac{1}{n}\,(\Lambda_\omega^r\zeta)\,\omega\bigg).\end{eqnarray*} This proves (\ref{eqn:star_Lefschetz_computed-components_r-1}).

\vspace{1ex}

Finally, (\ref{eqn:star_product_zeta_omega-powers_final}) follows at once by putting together (\ref{eqn:star_product_zeta_omega-powers}), (\ref{eqn:star_Lefschetz_computed-components_r-1}) and (\ref{eqn:star_Lefschetz_computed-components_r}).

\vspace{1ex}

$(2)$\, Let $v\in C^2_{p,\,q}(X,\,\C)$ be $\omega$-primitive. This means that $\Lambda_\omega v = 0$. Hence: \begin{eqnarray*}0=\Delta''_\omega\Lambda_\omega v = \Lambda_\omega\Delta''_\omega v,\end{eqnarray*} showing that $\Delta''_\omega v$ is $\omega$-primitive. We have used the commutation between $\Delta''_\omega$ and $\Lambda_\omega$ which holds since their adjoints, $\Delta''_\omega$ and $L_\omega$, commute due to $\omega$ being K\"ahler. This proves implication (\ref{eqn:primitive_Delta''_implication}).

The second statement follows at once from this implication, the commutation of $\Delta''_\omega$ with all $L_\omega^l$ and the uniqueness of the Lefschetz decomposition. Equivalence (\ref{eqn:harmonicity_form_prim-coordinates_equivalence}) follows at once from the above.

\vspace{1ex}

$(3)$\, Let $\alpha\in C^\infty_{m,\,m}(X,\,\C)$ be such that $\Delta''_\omega\alpha = 0$. Then, thanks to (\ref{eqn:harmonicity_form_prim-coordinates_equivalence}), we have $\Delta''_\omega\alpha_{prim}^{(m-1)} = 0$ and $\Delta''_\omega\alpha_{prim}^{(m)} = 0$. Meanwhile, (\ref{eqn:star_product_zeta_omega-powers}) ensures that $\star_\omega(\alpha\wedge\omega_{n-m-1})$ is a linear combination with constant coefficients of $\alpha_{prim}^{(m-1)}$ and $\alpha_{prim}^{(m)}$. We conclude that $\star_\omega(\alpha\wedge\omega_{n-m-1})$ is $\Delta''_\omega$-harmonic, thus proving implication (\ref{eqn:harmonicity_alpha_alpha-wedge-omega-power_implication}). The last statement follows from $\Delta_\omega = 2\,\Delta''_\omega$ and from $\ker\Delta_\omega = \ker d\cap\ker d^\star_\omega$. \hfill $\Box$

\begin{Cor}\label{Cor:zeta-prim-r-1_Lambda} In the setting of Lemma \ref{Lem:star_product_zeta_omega-powers}, we have the following equality of $(1,\,1)$-forms: \begin{eqnarray*}\zeta_{prim}^{(r-1)} = \frac{(n-r-1)!}{(n-2)!\,(r-1)!}\,(\Lambda_\omega^{r-1}\zeta)_{prim},\end{eqnarray*} where $(\Lambda_\omega^{r-1}\zeta)_{prim}$ is the $\omega$-primitive component of $\Lambda_\omega^{r-1}\zeta$ in its Lefschetz decomposition.

\end{Cor}

\noindent {\it Proof.} The Lefschetz decomposition with respect to $\omega$ of the $(1,\,1)$-form $\Lambda_\omega^{r-1}\zeta$ reads: \begin{eqnarray*}\Lambda_\omega^{r-1}\zeta = (\Lambda_\omega^{r-1}\zeta)_{prim} + \frac{1}{n}\,(\Lambda_\omega^r\zeta)\,\omega.\end{eqnarray*}

The contention follows at once by combining this with formula (\ref{eqn:star_Lefschetz_computed-components_r-1}).  \hfill $\Box$

\begin{Cor-Def}\label{Cor-Def:P_omega_def-formula} Let $(X,\,\omega)$ be a complex Hermitian manifold with $\mbox{dim}_\C X = n$ and let $m\in\{1,\dots , n\}$.

\vspace{1ex}

$(1)$\, The linear operator $Q_\omega:C^\infty_{m-1,\,m-1}(X,\,\R)\longrightarrow C^\infty_{1,\,1}(X,\,\R)$ defined by the first equality below is also expressed by the second equality: \begin{eqnarray}\label{eqn:P_omega_def}Q_\omega(u):= \star_\omega\bigg(i\partial\bar\partial u\wedge\omega_{n-m-1}\bigg) = \frac{1}{(m-1)!}\,\bigg(-\Lambda_\omega^{m-1}(i\partial\bar\partial u) + \frac{1}{m}\,(\Lambda_\omega^m(i\partial\bar\partial u))\,\omega\bigg).\end{eqnarray}

\vspace{1ex}

$(2)$\, If the metric $\omega$ is {\bf K\"ahler}, then for every $u\in C^\infty_{m-1,\,m-1}(X,\,\R)\cap\ker\partial^\star_\omega\cap\ker\bar\partial^\star_\omega$ we have: \begin{eqnarray*}Q_\omega(u)  = \frac{1}{(m-1)!}\,\bigg(-i\partial\bar\partial\Lambda_\omega^{m-1}u + (m-1)\,\Delta''_\omega\Lambda_\omega^{m-2} u - (\Delta''_\omega\Lambda_\omega^{m-1}u)\,\omega\bigg).\end{eqnarray*}

\vspace{1ex}

$(3)$\, If the metric $\omega$ is {\bf K\"ahler} and $\rho>0$ is an arbitrary Hermitian metric on $X$, then for every $u\in C^\infty_{m-1,\,m-1}(X,\,\R)\cap\ker\partial^\star_\omega\cap\ker\bar\partial^\star_\omega$ we have: \begin{eqnarray}\label{eqn:Lambda_rho-P_omega_formula}\nonumber\Lambda_\rho\bigg(Q_\omega(u)\bigg) & = & \frac{1}{(m-1)!}\,\bigg(\Delta''_\rho + \frac{m-n-1}{n}\, (\Lambda_\rho\omega)\,\Delta''_\omega + \bar\tau_\rho^\star\bar\partial\bigg)(\Lambda_\omega^{m-1} u) \\
  & + & \frac{(n-2)!}{(n-m)!}\,\Lambda_\rho\bigg(\Delta''_\omega(u_{prim}^{(m-2)})\bigg),\end{eqnarray}  with the understanding that the last term vanishes when $m=1$.

\end{Cor-Def}

\noindent {\it Proof.} $(1)$\, The second equality follows by taking $\zeta = i\partial\bar\partial u$ and $r=m$ in (\ref{eqn:star_product_zeta_omega-powers_final}).

\vspace{1ex}

$(2)$\, Putting together (\ref{eqn:P_omega_def}) and (\ref{eqn:Lambda-l-iddbar-u}), we get for every $u\in C^\infty_{m-1,\,m-1}(X,\,\R)\cap\ker\partial^\star_\omega\cap\ker\bar\partial^\star_\omega$: \begin{eqnarray*}Q_\omega(u)  = \frac{1}{(m-1)!}\,\bigg(-i\partial\bar\partial\Lambda_\omega^{m-1}u + (m-1)\,\Delta''_\omega\Lambda_\omega^{m-2} u + \frac{1}{m}\,(i\partial\bar\partial\Lambda_\omega^m u)\,\omega - (\Delta''_\omega\Lambda_\omega^{m-1}u)\,\omega\bigg),\end{eqnarray*} which is the stated formula since $\Lambda_\omega^m u = 0$ for bidegree reasons.

\vspace{1ex}

$(3)$\, Fix $u\in C^\infty_{m-1,\,m-1}(X,\,\R)\cap\ker\partial^\star_\omega\cap\ker\bar\partial^\star_\omega$. After taking $\Lambda_\rho$ in the equality stated under $(2)$, we get: \begin{eqnarray*}\Lambda_\rho\bigg(Q_\omega(u)\bigg)  = \frac{1}{(m-1)!}\,\bigg(-\Lambda_\rho\bigg(i\partial\bar\partial\Lambda_\omega^{m-1}u\bigg) - (\Lambda_\rho\omega)\,\Delta''_\omega(\Lambda_\omega^{m-1}u) + (m-1)\,\Lambda_\rho\bigg(\Delta''_\omega(\Lambda_\omega^{m-2} u)\bigg)\bigg).\end{eqnarray*}

We now apply formula (\ref{eqn:Laplacians_link_Hermitian}) with $\Lambda_\omega^{m-1}u$ (a function) in place of $u$. Since $\Lambda_\rho(\Lambda_\omega^{m-1} u) = 0$, $\partial_\rho^\star(\Lambda_\omega^{m-1} u) = 0$ and $\tau_\rho^\star(\Lambda_\omega^{m-1} u) = 0$ (for bidegree reasons), (\ref{eqn:Laplacians_link_Hermitian}) yields: \begin{eqnarray*}\Lambda_\rho\bigg(i\partial\bar\partial\Lambda_\omega^{m-1}u\bigg) = - \Delta''_\rho(\Lambda_\omega^{m-1} u) - \bar\tau_\rho^\star\bar\partial(\Lambda_\omega^{m-1} u).\end{eqnarray*}

Plugging this expression into the last formula for $\Lambda_\rho(Q_\omega(u))$, we get: \begin{eqnarray*}\Lambda_\rho\bigg(Q_\omega(u)\bigg) =  \frac{1}{(m-1)!}\,\bigg(\Delta''_\rho - (\Lambda_\rho\omega)\,\Delta''_\omega + \bar\tau_\rho^\star\bar\partial\bigg)(\Lambda_\omega^{m-1} u) + \frac{1}{(m-2)!}\,\Lambda_\rho\bigg(\Delta''_\omega(\Lambda_\omega^{m-2} u)\bigg),\end{eqnarray*} with the understanding that the last term vanishes when $m=1$. Now, $\Lambda_\omega^{m-2} u$ is a $(1,\,1)$-form, so its Lefschetz decomposition with respect to $\omega$ yields the first equality below: \begin{eqnarray*}\Lambda_\omega^{m-2} u = (\Lambda_\omega^{m-2} u)_{prim} + \frac{1}{n}\,(\Lambda_\omega^{m-1} u)\,\omega = \frac{(n-2)!\,(m-2)!}{(n-m)!}\,u_{prim}^{(m-2)} + \frac{1}{n}\,(\Lambda_\omega^{m-1} u)\,\omega,\end{eqnarray*} where the second equality follows from Corollary \ref{Cor:zeta-prim-r-1_Lambda} in which we take $\zeta:=u$ and $r:=m-1$. Plugging this expression for $\Lambda_\omega^{m-2} u$ into the last formula for $\Lambda_\rho(Q_\omega(u))$, we get (\ref{eqn:Lambda_rho-P_omega_formula}).  \hfill $\Box$

\begin{Cor}\label{Cor:PDE_two-prim-components} In the setting of Definition \ref{Def:equation}, the Monge-Amp\`ere-type equation (\ref{eqn:equation}) places constraints only on the $\omega$-primitive coordinates $u_{prim}^{(m-2)}$ of type $(1,\,1)$ and $u_{prim}^{(m-1)}$ of type $(0,\,0)$ of any solution $u\in C^\infty_{m-1,\,m-1}(X,\,\R)$.

  In other words, (\ref{eqn:equation}) can be seen as an equation whose possible solutions are pairs \begin{eqnarray*}\bigg(u_{prim}^{(m-2)},\,u_{prim}^{(m-1)}\bigg)\in C^\infty_{1,\,1}(X,\,\R)_{prim}\times C^\infty_{0,\,0}(X,\,\R).\end{eqnarray*}

\end{Cor}

\noindent {\it Proof.} With the notation of Cor-Def \ref{Cor-Def:P_omega_def-formula}, equation (\ref{eqn:equation}) reads: \begin{eqnarray*}\bigg(\star_\omega(\alpha\wedge\omega_{n-m-1}) + Q_\omega(u)\bigg)^n = dV.\end{eqnarray*} Meanwhile, $(2)$ of Cor-Def \ref{Cor-Def:P_omega_def-formula} shows that $Q_\omega(u)$ depends only on the pair $(\Lambda_\omega^{m-2} u,\,\Lambda_\omega^{m-1}u)\in C^\infty_{1,\,1}(X,\,\R)\times C^\infty_{0,\,0}(X,\,\R)$. Moreover, (\ref{eqn:star_Lefschetz_computed-components_r-1}) and (\ref{eqn:star_Lefschetz_computed-components_r}) show that this pair corresponds bijectively to the pair $(u_{prim}^{(m-2)},\,u_{prim}^{(m-1)})\in C^\infty_{1,\,1}(X,\,\R)_{prim}\times C^\infty_{0,\,0}(X,\,\R)$.  The conclusion follows.  \hfill $\Box$

\vspace{2ex}

We are now in a position to prove our first result (the uniqueness Theorem \ref{The:equation_uniqueness_introd}) concerning the new Monge-Amp\`ere-type equation introduced in Definition \ref{Def:equation}. It implies the following uniqueness (up to a constant multiple of $\omega^{m-1}$) result for the solutions of this equation. For the notation, see the definition of {\it $\omega$-primitive coordinates} given just above Lemma \ref{Lem:star_product_zeta_omega-powers}.

\begin{The}\label{The:equation_uniqueness} Let $(X,\,\omega)$ be a compact K\"ahler manifold with $\mbox{dim}_\C X = n$ and let $m\in\{1,\dots , n\}$. Fix a volume form $dV\in C^\infty_{n,\,n}(X,\,\R)$ with $dV>0$ and a form $\alpha\in C^\infty_{m,\,m}(X,\,\R)$ such that $d\alpha = 0$ and $\alpha>0$ strongly on $X$.

 If there exist $u_1 = \bigg((u_1)_{prim}^{(l)}\bigg)_{0\leq l\leq m-1},\,u_2= \bigg((u_2)_{prim}^{(l)}\bigg)_{0\leq l\leq m-1}\in C^\infty_{m-1,\,m-1}(X,\,\R)$ such that \begin{eqnarray*}\bigg[\star_\omega\bigg((\alpha + i\partial\bar\partial u_1)\wedge\omega_{n-m-1}\bigg)\bigg]^n = \bigg[\star_\omega\bigg((\alpha + i\partial\bar\partial u_2)\wedge\omega_{n-m-1}\bigg)\bigg]^n,\end{eqnarray*} satisfying the initial conditions: \begin{eqnarray*}\label{eqn:equation_initial-conditions_re-uniqueness}\alpha + i\partial\bar\partial u_j > 0 \hspace{1ex}\mbox{(strongly)} \hspace{5ex}  \mbox{and} \hspace{5ex} u_j\in\ker\partial^\star_\omega\cap\ker\bar\partial^\star_\omega,  \hspace{5ex} j=1,2,\end{eqnarray*} and, furthermore, satisfying the property: \begin{eqnarray}\label{eqn:equation_uniqueness_presciption}(u_1)_{prim}^{(l)} = (u_2)_{prim}^{(l)},  \hspace{5ex} l=0,\dots , m-2,\end{eqnarray} then there exists a constant $C\in\R$ such that \begin{eqnarray*}u_1 = u_2 + C\,\omega^{m-1}.\end{eqnarray*}

\end{The}

\noindent {\it Proof of Theorem \ref{The:equation_uniqueness_introd}, hence also of Theorem \ref{The:equation_uniqueness}.} Let $u_1,\,u_2\in C^\infty_{m-1,\,m-1}(X,\,\R)$ be solutions to equation (\ref{eqn:equation}). Then: \begin{eqnarray*}\star_\omega\bigg(i\partial\bar\partial(u_1-u_2)\wedge\omega_{n-m-1}\bigg)\wedge\rho^{n-1} = 0 \iff \Lambda_\rho\bigg(\star_\omega\bigg(i\partial\bar\partial(u_1-u_2)\wedge\omega_{n-m-1}\bigg)\bigg) = 0,\end{eqnarray*} where $\rho>0$ is the unique form $\rho\in C^\infty_{1,\,1}(X,\,\R)$ such that \begin{eqnarray*}\rho^{n-1} = \sum\limits_{s=1}^n\bigg[\star_\omega\bigg((\alpha + i\partial\bar\partial u_1)\wedge\omega_{n-m-1}\bigg)\bigg]^{n-s}\wedge\bigg[\star_\omega\bigg((\alpha + i\partial\bar\partial u_2)\wedge\omega_{n-m-1}\bigg)\bigg]^{s-1}>0.\end{eqnarray*} Indeed, $\rho$ is the (unique) $(n-1)^{st}$ root of the positive definite $(n-1,\,n-1)$-form on the r.h.s. above. The positive definiteness follows from the fact that $(\alpha + i\partial\bar\partial u_j)\wedge\omega_{n-m-1}>0$ at every point in $X$ for $j=1,2$ (thanks to the initial conditions (\ref{eqn:equation_initial-conditions})) and that products of positive definite $(1,\,1)$-forms are strictly strongly positive.

Thus, whenever $u_1,\,u_2\in C^\infty_{m-1,\,m-1}(X,\,\R)$ are solutions to equation (\ref{eqn:equation}), we have: \begin{eqnarray}\label{eqn:uniqueness_proof_P_equality}\Lambda_\rho\bigg(Q_\omega(u_1-u_2)\bigg) = 0,\end{eqnarray} where the linear operator $Q_\omega:C^\infty_{m-1,\,m-1}(X,\,\R)\longrightarrow C^\infty_{1,\,1}(X,\,\R)$ was introduced in Corollary and Definition \ref{Cor-Def:P_omega_def-formula}.

We conclude from the combined (\ref{eqn:uniqueness_proof_P_equality}) and (\ref{eqn:Lambda_rho-P_omega_formula}) that, whenever $u_1,\,u_2\in C^\infty_{m-1,\,m-1}(X,\,\R)$ are solutions to equation (\ref{eqn:equation}) such that $(u_1)_{prim}^{(m-2)} = (u_2)_{prim}^{(m-2)}$, we have: \begin{eqnarray}\label{eqn:uniqueness_proof_P_equality_initial-condition}\bigg(\Delta''_\rho + \frac{m-n-1}{n}\, (\Lambda_\rho\omega)\,\Delta''_\omega + \bar\tau_\rho^\star\bar\partial\bigg)\bigg(\Lambda_\omega^{m-1} u_1 - \Lambda_\omega^{m-1} u_2\bigg) = 0.\end{eqnarray} Since the second-order differential operator on the left is elliptic with no non-zero zero$^{th}$ order term (it is actually a Laplacian, as was proved with slightly different coefficients in [Pop15, proof of Proposition 6.1.]) and since the manifold $X$ is compact, the function $\Lambda_\omega^{m-1} u_1 - \Lambda_\omega^{m-1} u_2$ must be constant. Thanks to (\ref{eqn:star_Lefschetz_computed-components_r}) (applied with $r=m-1$ and $\zeta:=u_1-u_2$), this is equivalent to the function $(u_1)_{prim}^{(m-1)} - (u_2)_{prim}^{(m-1)}$ being constant on $X$.   \hfill $\Box$

\section{An application of the new Monge-Amp\`ere-type PDE}\label{section:application_PDE}

We start with a pointwise inequality that can be viewed as a higher-degree analogue of inequality $(7)$ in Lemma $3.1$ of [Pop16].

\begin{Lem}\label{Lem:pointwise_ineq_form-products} Let $(X,\,\omega)$ be a complex Hermitian $n$-dimensional manifold and let $m\in\{1,\dots , n-1\}$.

  Then, for any $(1,\,1)$-form $\rho>0$, any $(m,\,m)$-form $\beta\geq 0$ (strongly) and any $(n-m,\,n-m)$-form $\Omega\geq 0$ (weakly) on $X$, the following inequality holds at every point of $X$: \begin{eqnarray}\label{eqn:pointwise_ineq_form-products}\frac{\rho_m\wedge\Omega}{\omega_n}\,\frac{\rho_{n-m}\wedge\beta}{\rho_n}\geq\frac{\beta\wedge\Omega}{\omega_n},\end{eqnarray} with the usual notation $a_p:=a^p/p!$ for every $(1,\,1)$-form and every $p$.

\end{Lem}

\begin{proof} Since the claim is a pointwise inequality, it suffice to prove it at an arbitrary point $x\in X$ that we now fix. Since $\beta$ is strongly semi-positive, there exist local holomorphic coordinates $z_1,\dots , z_n$ centred at $x$ such that $\rho$ and $\beta$ are simultaneously diagonalised at $x$ as: \begin{eqnarray*}\rho(x) = \sum\limits_{j=1}^n\rho_j\,idz_j\wedge d\bar{z}_j  \hspace{5ex}\mbox{and}\hspace{5ex} \beta(x) = \sum\limits_{|J|=m}\beta_J\,i^{m^2}\,dz_J\wedge d\bar{z}_J,\end{eqnarray*} where $J=(j_1,\dots , j_m)$ are multi-indices of length $m$ and $dz_J:=dz_{j_1}\wedge\cdots\wedge dz_{j_m}$ whenever $J=(j_1,\dots , j_m)$ with $1\leq j_1<\dots <j_m\leq n$. The (semi-)positivity assumptions on $\rho$ and $\beta$ mean that $\rho_j>0$ for all $j$ and $\beta_J\geq 0$ for all $J$.

  Let $\Omega = \sum\limits_{|K| = |L| = n-m}\Omega_{K\bar{L}}\,i^{(n-m)^2}\,dz_K\wedge d\bar{z}_L$ be the expression of $\Omega$ in the chosen coordinates in a neighbourhood of $x$. The weak semi-positivity assumption on $\Omega$ implies that $\Omega_{L\bar{L}}\geq 0$ for every $L$.

  At $x$, we get: \begin{eqnarray*}\rho_m\wedge\Omega = \bigg(\sum\limits_{|J|=m}\rho_J\,i^{m^2}\,dz_J\wedge d\bar{z}_J\bigg)\wedge\sum\limits_{|K| = |L| = n-m}\Omega_{K\bar{L}}\,i^{(n-m)^2}\,dz_K\wedge d\bar{z}_L = \bigg(\sum\limits_{|J|=m}\rho_J\,\Omega_{C_J\overline{C_J}}\bigg)\,dV_n,\end{eqnarray*} where $\rho_J:=\rho_{j_1}\cdots\rho_{j_m}$ whenever $J=(j_1,\dots , j_m)$ with $1\leq j_1<\dots <j_m\leq n$, while $C_J:=\{1,\dots , n\}\setminus J$ is the multi-index (of length $n-m$) complementary to $J$ and $dV_n:=idz_1\wedge d\bar{z}_1\wedge\dots\wedge idz_n\wedge d\bar{z}_n$ is the standard volume form in the chosen coordinates defined in a neighbourhhod of $x$.

  Similarly, we get at $x$: \begin{eqnarray*}\rho_{n-m}\wedge\beta = \bigg(\sum\limits_{|K|=n-m}\rho_K\,i^{(n-m)^2}\,dz_K\wedge d\bar{z}_K\bigg)\wedge\sum\limits_{|J|=m}\beta_J\,i^{m^2}\,dz_J\wedge d\bar{z}_J = \bigg(\sum\limits_{|K|=n-m}\rho_K\,\beta_{C_K}\bigg)\,dV_n,\end{eqnarray*} hence \begin{eqnarray*}\frac{\rho_{n-m}\wedge\beta}{\rho_n} = \frac{\sum\limits_{|K|=n-m}\rho_K\,\beta_{C_K}}{\rho_1\dots\rho_n} = \sum\limits_{|K|=n-m}\frac{\beta_{C_K}}{\rho_{C_K}} = \sum\limits_{|I|=m}\frac{\beta_I}{\rho_I}.\end{eqnarray*}

  Since $\beta\wedge\Omega = \bigg(\sum\limits_{|J|=m}\beta_J\,\Omega_{C_J\overline{C_J}}\bigg)\,dV_n$ at $x$, we get the following equivalences at $x$, the first of which being obtained by replacing $\omega_n$ in the two denominators of (\ref{eqn:pointwise_ineq_form-products}) with $dV_n$: \begin{eqnarray*}(\ref{eqn:pointwise_ineq_form-products}) & \iff & \frac{\rho_m\wedge\Omega}{dV_n}\,\frac{\rho_{n-m}\wedge\beta}{\rho_n}\geq\frac{\beta\wedge\Omega}{dV_n} \iff \bigg(\sum\limits_{|J|=m}\rho_J\,\Omega_{C_J\overline{C_J}}\bigg)\,\bigg(\sum\limits_{|I|=m}\frac{\beta_I}{\rho_I}\bigg) \geq \sum\limits_{|J|=m}\beta_J\,\Omega_{C_J\overline{C_J}} \\   & \iff & \sum\limits_{|I|=|J|=m}\frac{\rho_J}{\rho_I}\,\beta_I\,\Omega_{C_J\overline{C_J}} \geq \sum\limits_{|J|=m}\beta_J\,\Omega_{C_J\overline{C_J}}.\end{eqnarray*} This last inequality holds since all the terms in these two sums are non-negative (thanks to our (semi-)positivity assumptions on $\rho$, $\beta$ and $\Omega$) and the sum on the left (whose terms corresponding to $I=J$ are precisely the terms of the sum on the right) contains more terms than the sum on the right. \end{proof}

\vspace{2ex}

Next, we specify the existence result, with appropriate initial conditions, that will be needed and whose study will hopefully be taken up in future work.

\begin{Def}\label{Def:existence_equation_initial-cond} Let $(X,\,\omega)$ be a compact {\bf K\"ahler} manifold with $\mbox{dim}_\C X = n$ and let $m\in\{1,\dots , n\}$.

  We say that equation (\ref{eqn:equation}) is {\bf solvable} if for every $\alpha\in C^\infty_{m,\,m}(X,\,\R)$ such that $d\alpha = 0$ and $\alpha>0$ (strongly) and every volume form $dV\in C^\infty_{n,\,n}(X,\,\R)$ with $dV>0$, there exists a constant $c>0$ and a form $u\in C^\infty_{m-1,\,m-1}(X,\,\R)$ such that \begin{eqnarray}\label{eqn:equation_re}\bigg[\star_\omega\bigg((\alpha + i\partial\bar\partial u)\wedge\omega_{n-m-1}\bigg)\bigg]^n = c\,dV\end{eqnarray} and such that $u$ satisfies the following initial conditions: \begin{eqnarray}\label{eqn:equation_initial-conditions_re}\alpha + i\partial\bar\partial u > 0 \hspace{1ex}\mbox{(strongly)} \hspace{5ex}  \mbox{and} \hspace{5ex} u\in\ker\partial^\star_\omega\cap\ker\bar\partial^\star_\omega  \hspace{5ex}  \mbox{and} \hspace{5ex} \Lambda_\omega^{m-2}(\Delta''_\omega u) = 0.\end{eqnarray}

\end{Def}

\vspace{2ex}

We can now give the main result of this section, an analogue of Theorem 1.1 in [Pop16], subject to the solvability of our Monge-Amp\`ere-type pde (\ref{eqn:equation}). It is slightly more general than Theorem \ref{The:current_existence_introd}.

\begin{The}\label{The:current_existence} Let $(X,\,\omega)$ be a compact {\bf K\"ahler} manifold with $\mbox{dim}_\C X = n$ and let $m\in\{1,\dots , n-1\}$. Suppose that equation (\ref{eqn:equation}) is {\bf solvable} in the sense of Definition \ref{Def:existence_equation_initial-cond}.

  Suppose there exist forms $\alpha, \beta\in C^\infty_{m,\,m}(X,\,\R)$ such that:

\vspace{1ex}

(i)\, $d\alpha = 0$, $\alpha>0$ (strongly) and $\Delta''_\omega\alpha_\omega =0$, where $\alpha_\omega:=\star_\omega\bigg(\alpha\wedge\omega_{n-m-1}\bigg)\in C^\infty_{1,\,1}(X,\,\R)$;

\vspace{1ex}

(ii)\, $\beta\geq C\,\omega_m$ (strongly) for some constant $C>0$ and $\partial\bar\partial\beta = 0$;

\vspace{1ex}

(iii)\, $\displaystyle \frac{1}{(n-m)!}\,\int\limits_X(\alpha_\omega)^{n-m}\wedge\beta < \frac{1}{n!}\,\int\limits_X(\alpha_\omega)^n$.

\vspace{1ex}

Then, there exists a real current $T$ of bidegree $(m,\,m)$ on $X$ with the following properties: \begin{eqnarray*}(a)\, T\geq \delta\,(\alpha_\omega)_m \hspace{3ex}\mbox{for some constant}\hspace{1ex} \delta>0; \hspace{2ex} (b)\, \partial\bar\partial T = 0; \hspace{2ex} (c)\, T\in\bigg[(\alpha_\omega)_m - \beta\bigg]_A,\end{eqnarray*} where $[\,\cdot\,]_A$ stands for the Aeppli cohomology class of the specified form.

Moreover, if $d\beta=0$, the above current $T$ can be found such that $dT = 0$ and $T$ lies in the Bott-Chern cohomology class $[(\alpha_\omega)_m - \beta]_{BC}\in H^{m,\,m}_{BC}(X,\,\R)$.

\end{The}

Note that the strong positivity assumption on $\alpha$ implies the positivity of the $(1,\,1)$-form $\alpha_\omega$, hence the strong positivity of the $(m,\,m)$-form $(\alpha_\omega)_m$. In particular, the above conclusion (a) implies a form of strict positivity for $T$, analogous at the level of $(m,\,m)$-currents to the strict positivity required  in bidegree $(1,\,1)$ of K\"ahler currents. Thus, Theorem \ref{The:current_existence} produces a kind of $m$-K\"ahler current lying in the Aeppli (resp. Bott-Chern, depending on whether $\beta$ is supposed $\partial\bar\partial$-closed or $d$-closed) cohomology class of $(\alpha_\omega)_m - \beta$.

Also note that hypothesis (iii) of Theorem \ref{The:current_existence} implies the following inequality of cohomology intersection numbers: \begin{eqnarray}\label{eqn:intersection-numbers_coh_ineq}\frac{1}{(n-m)!}\,[\alpha_\omega]_{BC}^{n-m}.\,[\beta]_A < \frac{1}{n!}\,[\alpha_\omega]^n_{BC}.\end{eqnarray}

\noindent {\it Proof of Theorem \ref{The:current_existence}.} 

$\bullet$ We wish to prove the existence of a constant $\delta>0$ and of a real current $S\in{\cal D}^{'m-1,\,m-1}(X,\,\R)$ of bidegree $(m-1,\,m-1)$ on $X$ such that $T:=(\alpha_\omega)_m - \beta + i\partial\bar\partial S\geq\delta\,(\alpha_\omega)_m$ (strongly). Thanks to the duality Lemma \ref{Lem:duality_generalised-Lamari}, this is equivalent to proving the existence of a constant $\delta>0$ such that
\begin{eqnarray*}\int_X\bigg((\alpha_\omega)_m - \beta\bigg)\wedge\Omega\geq\delta\,\int_X(\alpha_\omega)_m\wedge\Omega \hspace{3ex}\mbox{or equivalently}\hspace{3ex} (1-\delta)\,\int_X(\alpha_\omega)_m\wedge\Omega\geq\int_X\beta\wedge\Omega
\end{eqnarray*}
for all forms $\Omega\in C^\infty_{n-m,\,n-m}(X,\,\R)$ satisfying the conditions $\partial\bar\partial\Omega = 0$ and $\Omega>0$ (weakly).

  We will reason by contradiction. Suppose this is not true. Then, there exist a sequence of constants $\delta_k\downarrow 0$ as $k\to\infty$ and a sequence of forms $\Omega_k\in C^\infty_{n-m,\,n-m}(X,\,\R)$ such that $\partial\bar\partial\Omega_k = 0$, $\Omega_k>0$ (weakly) and \begin{eqnarray}\label{eqn:contradiction-ineq_normalisation}(1-\delta_k)\,\int_X(\alpha_\omega)_m\wedge\Omega_k < \int_X\beta\wedge\Omega_k =1 \hspace{5ex}\mbox{for all}\hspace{2ex} k\in\N,\end{eqnarray} where the last equality is a normalisation of the forms $\Omega_k$ (which can be achieved by multiplying each of these forms by an appropriate positive constant).

  Since $\beta\geq C\,\omega_m$ (by hypothesis (ii)), the above normalisation of the weakly positive forms $\Omega_k$ implies their uniform boundedness in mass, hence the existence of a subsequence (still denoted by $(\Omega_k)_{k\in\N}$) converging in the weak topology of currents to some weakly positive current $\Omega_\infty$ of bidegree $(n-m,\,n-m)$: $\Omega_k\longrightarrow\Omega_\infty\geq 0$ (weakly) as $k\to\infty$. We get, after letting $k\to\infty$ in (\ref{eqn:contradiction-ineq_normalisation}): \begin{eqnarray}\label{eqn:ineq_leq_1}\int_X(\alpha_\omega)_m\wedge\Omega_\infty\leq 1.\end{eqnarray}

  $\bullet$ Meanwhile, thanks to the solvability assumption in the sense of Definition \ref{Def:existence_equation_initial-cond} on equation (\ref{eqn:equation}), for every $k\in\N$, there exists a constant $c_k>0$ and a form $u_k\in C^\infty_{m-1,\,m-1}(X,\,\R)$ such that \begin{eqnarray}\label{eqn:equation_u_k}\frac{1}{n!}\,\bigg[\star_\omega\bigg((\alpha + i\partial\bar\partial u_k)\wedge\omega_{n-m-1}\bigg)\bigg]^n = c_k\,\beta\wedge\Omega_k
  \end{eqnarray} and such that \begin{eqnarray}\label{eqn:equation_initial-conditions_re_u_k}\widetilde\alpha_k:=\alpha + i\partial\bar\partial u_k > 0 \hspace{1ex}\mbox{(strongly)} \hspace{3ex}  \mbox{and} \hspace{3ex} u_k\in\ker\partial^\star_\omega\cap\ker\bar\partial^\star_\omega  \hspace{3ex}  \mbox{and} \hspace{3ex} \Lambda_\omega^{m-2}(\Delta''_\omega u_k) = 0.\end{eqnarray} In terms of $\widetilde\alpha_k$, (\ref{eqn:equation_u_k}) can be reworded as \begin{eqnarray}\label{eqn:equation_u_k_tilde}\bigg((\widetilde\alpha_k)_\omega\bigg)_n = c_k\,\beta\wedge\Omega_k,  \hspace{6ex} k\in\N.\end{eqnarray}

    $\bullet$ We get the following inequalities: \begin{eqnarray}\label{eqn:ineq_C-S+M-A+c_k}\nonumber & & \bigg(\int\limits_X\bigg((\widetilde\alpha_k)_\omega\bigg)_m\wedge\Omega_k\bigg)\cdot\bigg(\int\limits_X\bigg((\widetilde\alpha_k)_\omega\bigg)_{n-m}\wedge\beta\bigg) \\
\nonumber    & \geq & \bigg(\int\limits_X\sqrt{\frac{((\widetilde\alpha_k)_\omega)_m\wedge\Omega_k}{\omega_n}\cdot\frac{((\widetilde\alpha_k)_\omega)_{n-m}\wedge\beta}{((\widetilde\alpha_k)_\omega)_n}}\cdot\sqrt{\frac{((\widetilde\alpha_k)_\omega)_n}{\omega_n}}\,\omega_n\bigg)^2 \geq \bigg(\int\limits_X\sqrt{\frac{\beta\wedge\Omega_k}{\omega_n}}\cdot\sqrt{c_k\,\frac{\beta\wedge\Omega_k}{\omega_n}}\,\omega_n\bigg)^2 \\
& = & c_k\,\bigg(\int\limits_X\beta\wedge\Omega_k\bigg)^2 = c_k,  \hspace{5ex} k\in\N,\end{eqnarray} where the first inequality is an application of the classical H\"older inequality, the second inequality is an application of the pointwise inequality (\ref{eqn:pointwise_ineq_form-products}) with $\rho:=(\widetilde\alpha_k)_\omega$ and $\Omega:=\Omega_k$ combined with equality (\ref{eqn:equation_u_k_tilde}) (the result of the solvability of equation (\ref{eqn:equation})), while the last equality follows from the normalisation (\ref{eqn:contradiction-ineq_normalisation}) of the forms $\Omega_k$.

    On the other hand, for every $k\in\N$ we have: \begin{eqnarray}\label{eqn:alpha-tilde-k_iddbar}\nonumber(\widetilde\alpha_k)_\omega & = & \star_\omega\bigg((\alpha + i\partial\bar\partial u_k)\wedge\omega_{n-m-1}\bigg) = \star_\omega\bigg(\alpha\wedge\omega_{n-m-1}\bigg) + Q_\omega(u_k) \\
      & = & \alpha_\omega - \frac{1}{(m-1)!}\,i\partial\bar\partial\Lambda_\omega^{m-1}u_k,\end{eqnarray} where the third equality follows from part $(2)$ of Corollary and Definition \ref{Cor-Def:P_omega_def-formula} thanks to the initial conditions (\ref{eqn:equation_initial-conditions_re_u_k}) satisfied by $u_k$. Indeed, $\Delta''_\omega$ commutes with $L_\omega:=\omega\wedge\cdot\,$ thanks to the metric $\omega$ being K\"ahler, hence their adjoints, $\Delta''_\omega$ and $\Lambda_\omega$, commute. Thus, $\Delta''_\omega\Lambda_\omega^{m-2} u_k = \Lambda_\omega^{m-2}\Delta''_\omega u_k =0$, which implies $0 = \Lambda_\omega^{m-1}\Delta''_\omega u_k = \Delta''_\omega\Lambda_\omega^{m-1} u_k$.

    We conclude from (\ref{eqn:alpha-tilde-k_iddbar}), from $d\alpha_\omega = 0$ (which is a consequence of the hypothesis $\Delta''_\omega\alpha_\omega = 0$, that is, in turn, equivalent to $\Delta_\omega\alpha_\omega = 0$ since $\omega$ is K\"ahler and is again equivalent to $\alpha_\omega\in\ker d\cap\ker d_\omega^\star$), from $\partial\bar\partial\Omega_k = 0$ and from the Stokes theorem that $u_k$ disappears from the integrals featuring on the first and last lines of (\ref{eqn:ineq_C-S+M-A+c_k}) and in the definition of $c_k$. In other words, we have: \begin{eqnarray*}\int\limits_X\bigg((\widetilde\alpha_k)_\omega\bigg)_m\wedge\Omega_k =\int\limits_X(\alpha_\omega)_m\wedge\Omega_k, \hspace{5ex}  \int\limits_X\bigg((\widetilde\alpha_k)_\omega\bigg)_{n-m}\wedge\beta = \int\limits_X(\alpha_\omega)_{n-m}\wedge\beta \end{eqnarray*} and $\displaystyle c_k = \int\limits_X \bigg((\widetilde\alpha_k)_\omega\bigg)_n = \int\limits_X(\alpha_\omega)_n$.

    Thus, (\ref{eqn:ineq_C-S+M-A+c_k}) reduces to the inequality: \begin{eqnarray*}\bigg(\int\limits_X(\alpha_\omega)_m\wedge\Omega_k\bigg)\cdot\bigg(\int\limits_X(\alpha_\omega)_{n-m}\wedge\beta\bigg) \geq \int\limits_X(\alpha_\omega)_n,  \hspace{5ex} k\in\N.\end{eqnarray*} Note that the second and third quantities above are independent of $k$, while for the first we saw in (\ref{eqn:contradiction-ineq_normalisation}) and (\ref{eqn:ineq_leq_1}) that we have: $\lim\limits_{k\to\infty}\int_X(\alpha_\omega)_m\wedge\Omega_k = \int_X(\alpha_\omega)_m\wedge\Omega_\infty\leq 1$. Therefore, letting $k\to\infty$ in the above inequality, we get:  \begin{eqnarray*}\int\limits_X(\alpha_\omega)_{n-m}\wedge\beta \geq \int\limits_X(\alpha_\omega)_n.\end{eqnarray*} This contradicts hypothesis (iii) and we are done. \hfill $\Box$

\vspace{2ex}

When $m=n-1$, Theorem \ref{The:current_existence} can be specified further since the linear map \begin{eqnarray*}\omega_{n-2}\wedge\cdot\,:\Lambda^{1,\,1}T^\star X\longrightarrow\Lambda^{n-1,\,n-1}T^\star X\end{eqnarray*} is bijective at every point of $X$.

  When $d\beta=0$, the $(n-1,\,n-1)$-current $T$ given by the conclusion is $d$-closed, hence also $\bar\partial$-closed, so it defines a Dolbeault cohomology class $[T]_{\bar\partial} = [(\alpha_\omega)_{n-1} - \beta]_{\bar\partial}\in H^{n-1,\,n-1}_{\bar\partial}(X,\,\R)$. Let $\zeta\in C^\infty_{n-1,\,n-1}(X,\,\R)$ be the $\Delta''_\omega$-harmonic representative of this class. By the bijectivity of the above pointwise map, there is a unique form $\eta\in C^\infty_{1,\,1}(X,\,\R)$ such that $\zeta = \omega_{n-2}\wedge\eta$. Taking $\Delta''_\omega$, we get: \begin{eqnarray*}0=\Delta''_\omega\zeta = \omega_{n-2}\wedge\Delta''_\omega\eta,\end{eqnarray*} the last equality being a consequence of the commutation of $\Delta''_\omega$ with $\omega_{n-2}\wedge\cdot$ due to $\omega$ being K\"ahler. The bijectivity of the above pointwise map implies $\Delta''_\omega\eta = 0$. (This is equivalent to $\Delta_\omega\eta = 0$, since $\omega$ is K\"ahler, which implies that $d\eta = 0$, so $\eta$ represents a Bott-Chern cohomology class of bidegree $(1,\,1)$.) In particular, $\eta$ represents a Dolbeault cohomology class $[\eta]_{\bar\partial}\in H^{1,\,1}_{\bar\partial}(X,\,\R)$ and this class is the inverse image of \begin{eqnarray*}[\omega_{n-2}\wedge\eta]_{\bar\partial} = [\zeta]_{\bar\partial} = [T]_{\bar\partial} = [(\alpha_\omega)_{n-1} - \beta]_{\bar\partial}\in H^{n-1,\,n-1}_{\bar\partial}(X,\,\R)\end{eqnarray*} under the Hard Lefschetz isomorphism \begin{eqnarray*}H^{1,\,1}_{\bar\partial}(X,\,\R)\ni[a]_{\bar\partial}\longmapsto[\omega_{n-2}\wedge a]_{\bar\partial}\in H^{n-1,\,n-1}_{\bar\partial}(X,\,\R).\end{eqnarray*}

  Since, by the $\partial\bar\partial$-lemma, $T = \omega_{n-2}\wedge\eta + i\partial\bar\partial\chi$ for some real current $\chi$ of bidegree $(n-2,\,n-2)$, conclusion (a) of Theorem \ref{The:current_existence} reads: \begin{eqnarray*}T = \omega_{n-2}\wedge\eta + i\partial\bar\partial\chi\geq \delta\,(\alpha_\omega)_{n-1}\end{eqnarray*} for some constant $\delta>0$, where the $(1,\,1)$-form $\alpha_\omega:=\star_\omega\alpha$ is positive definite at every point of $X$.

  This proves, in the language of Definition \ref{Def:m-psef}, that the class $[\eta]_{BC}\in H^{1,\,1}_{BC}(X,\,\R)$ is $[\omega]$-$(n-1)$-{\bf big}.

\vspace{2ex}

  We have thus proved Corollary \ref{Cor:n-1_current_existence_big_introd} stated in the introduction.

\vspace{3ex}

\noindent {\bf References.} \\



\noindent [BDPP13]\, S. Boucksom, J.-P. Demailly, M. Paun, T. Peternell --- {\it The Pseudo-effective Cone of a Compact K\"ahler Manifold and Varieties of Negative Kodaira Dimension} --- J. Alg. Geom. {\bf 22} (2013) 201-248.







\vspace{1ex}

\noindent [Dem92]\, J.-P. Demailly --- {\it Regularization of Closed Positive Currents and Intersection Theory} --- J. Alg. Geom., {\bf 1} (1992), 361-409.





\vspace{1ex}

\noindent [Die06]\, N. Q. Dieu --- {\it $q$-Plurisubharmonicity and $q$-Pseudoconvexity in $\C^n$} --- Publ. Mat., Barc. {\bf 50} (2006), no. 2, p. 349-369.

\vspace{1ex}

\noindent [Din22]\, S. Dinew --- {\it $m$-subharmonic and $m$-plurisubharmpnic functions: on two problems of Sadullaev} --- Ann. Fac. Sci. Toulouse Math. (6) {\bf 31} (2022), no. 3, 995-1009.

\vspace{1ex}

\noindent [DP25]\, S. Dinew, D. Popovici --- {\it  $m$-Positivity and Regularisation} --- arXiv:2510.25639v1 [math.DG]







\vspace{1ex}

\noindent [HL13]\, F.R. Harvey, H.B. Lawson --- {\it p-Convexity, p-Plurisubharmonicity and the Levi Problem} --- Indiana Univ. Math. J. {\bf62} (2013), no. 1, 149-169.





\vspace{1ex}

\noindent [Lam99]\, A. Lamari --- {\it Courants k\"ahl\'eriens et surfaces compactes} --- Ann. Inst. Fourier, Grenoble, {\bf 49}, 1 (1999), 263-285.





\vspace{1ex}

\noindent [Pop15]\, D. Popovici --- {\it Aeppli Cohomology Classes Associated with Gauduchon Metrics on Compact Complex Manifolds} --- Bull. Soc. Math. France {\bf 143} (3), (2015), p. 1-37.

\vspace{1ex}

\noindent [Pop16]\, D. Popovici --- {\it Sufficient Bigness Criterion for Differences of Two Nef Classes} --- Math. Ann. {\bf 364} (2016), 649-655.

\vspace{1ex}

\noindent [Ver10]\, M. Verbitsky --- {\it Plurisubharmonic Functions in Calibrated Geometry and q-Convexity} --- Math. Z. {\bf 264} (2010), no. 4, p. 939-957.

\vspace{1ex}

\noindent [Voi02]\, C. Voisin --- {\it Hodge Theory and Complex Algebraic Geometry. I.} --- Cambridge Studies in Advanced Mathematics, 76, Cambridge University Press, Cambridge, 2002.



\vspace{6ex}

\noindent Department of Mathematics and Computer Science     \hfill Institut de Math\'ematiques de Toulouse,

\noindent Jagiellonian University     \hfill  Universit\'e de Toulouse,

\noindent 30-409 Krak\'ow, Ul. Lojasiewicza 6, Poland    \hfill  118 route de Narbonne, 31062 Toulouse, France

\noindent  Email: Slawomir.Dinew@im.uj.edu.pl                \hfill     Email: popovici@math.univ-toulouse.fr

\end{document}